\documentclass[12pt]{article}

\usepackage{amsmath,amssymb,amsthm,textcomp,mathtools}
\usepackage{amsfonts,graphicx}
\usepackage[mathscr]{eucal}
\usepackage{color}
\usepackage{csquotes}
\usepackage{enumitem}
\usepackage{authblk} 
\usepackage{comment}

\numberwithin{equation}{section}

\theoremstyle{definition}

\numberwithin{equation}{section}

\newtheorem{theorem}{\bf Theorem}[section]
\newtheorem{remark}{\bf Remark}[section]
\newtheorem{proposition}{Proposition}[section]
\newtheorem{lemma}{Lemma}[section]
\newtheorem{corollary}{Corollary}[section]
\newtheorem{example}{Example}[section]

\newtheoremstyle
{remarkstyle}
{}
{11pt}
{}
{}
{\bfseries}
{:}
{     }
{\thmname{#1} \thmnumber{#2} }

\theoremstyle{remarkstyle}

\begin{document}

    \title{Convergence of random sums in non-commutative probability 
    }
    \author{Arup Bose\thanks{bosearu@gmail.com, \ Research  
supported by J.C.~Bose National Fellowship, JBR/2023/000023 from Anusandhan National Research Foundation, Govt. of India.}}

\author{Pradeep Vishwakarma\thanks{vishwakarmapr.rs@gmail.com, Research  
supported by National Post Doctoral Fellowship, PDF/2025/000076 from Anusandhan National Research Foundation, Govt. of India.}}
\affil{Theoretical Statistics and Mathematics Unit\\ Indian Statistical Institute, Kolkata, 700108, India}
	\date{\today}	
	\maketitle
	\begin{abstract} 
    Random sums of independent random variables have been extensively studied in classical probability theory. We consider random sums of self-adjoint variables from a non-commutative probability space, and establish several $*$-convergence results. In particular, we show that the joint $*$-convergence of the standardized random sum of identically distributed self-adjoint variables and the standardized stopping random variable (rv) is equivalent to the convergence of all moments of the stopping rv together with the convergence of the ratio of its mean to its variance.
    We obtain central limit theorems for the random sums of free, independent and half independent self-adjoint variables with both deterministic and random scaling.  Furthermore, we derive some scaling $*$-convergence limits for randomly indexed self-adjoint variables.
   \end{abstract}  
\\


\section{Introduction}\label{sec1} 
Let $\{Y_n\}$ be independent identically distributed (iid) random variables (rvs), independent of non-negative integer-valued rvs $\{N_n\}$, 
$n \geq 1$.
As $n\rightarrow\infty$, the asymptotic behavior of the partial random sums $Y_1+\dots+Y_{N_n}$ has been analyzed extensively from various perspectives,  dating back to early works of \cite{Anscombe1952}, \cite{Robbins1948}, \cite{Reny1960} and others. Relatively recent works include 
\cite{Barbour2006, Gnedenko1996, Silvestrov2004, Watanabe2008} and the references therein. We quote the following result from  \cite{Robbins1948}:  
Let 
$Z$ be a standard Gaussian rv, independent of $\{N_n\}$, and define
\begin{align}
\tilde{N}_n:&=\dfrac{N_n-\mathbb{E}[N_n]}{\sqrt{\mathbb{V}\mathrm{ar}[N_n]}},\label{Nndefn}\\
S_{N_{n}}&=Y_1+\cdots + Y_{N_{n}}, \label{SNndefn}\\
\sigma_n^2&:= \mathbb{V}\mathrm{ar}[S_{N_n}]=\mathbb{E}[N_n]\mathbb{V}\mathrm{ar}[Y_1]+\mathbb{V}\mathrm{ar}[N_n](\mathbb{E}[Y_1])^2,\label{sigmandefn}
\\
\delta_n^2&:=\mathbb{V}\mathrm{ar}[N_n][\mathbb{E}Y_1]^2/
\sigma_n^2
\ \ (\text{provided it is well-defined}).\nonumber
 \end{align}
\begin{theorem}\label{prop:classical}
    Let $\{Y_n\}$ be iid 
    with $0<\mathbb{E}[Y_1^2]<\infty$. Let $\{N_n\}$ be non-negative integer-valued, independent of $\{Y_n\}$,  with finite  variance. Let $\sigma_n^2$ be as in (\ref{sigmandefn}).\vskip3pt
    
   \noindent (i) If $\mathbb{E}[Y_1]\neq0$, and
    as $n\rightarrow\infty$, $\sigma_n^2\to \infty$,
    along with 
    $\sqrt{\mathbb{V}\mathrm{ar}[N_n]}=o(\sigma_n^2)$,
    then 
    \begin{equation}\label{introre1}
    \mathbb{E}\exp\Big[\iota t(\frac{S_{N_n}-\mathbb{E}[N_nY_1]}
    {\sigma_n})\Big]=\mathbb{E}\exp[\iota\delta_nt\tilde{N}_n]\mathbb{E}\exp[\iota t\sqrt{1-\delta_n^2}Z]+o(1),\ t\in\mathbb{R},
\end{equation}
where $\tilde{N}_n$ and $S_{N_{n}}$ are as in (\ref{Nndefn}) and (\ref{SNndefn}), respectively. \vskip5pt

\noindent (ii)  If $\mathbb{E}[Y_n]=0$, and as $n\rightarrow\infty$, $\mathbb{E}[N_n]\rightarrow\infty$, $\lim_{n\rightarrow\infty}\frac{\sqrt{\mathbb{V}\mathrm{ar}[N_n]}}{\mathbb{E}[N_n]}=:\xi\in(0,\infty)$, and 
$\tilde{N}_n$ converges weakly to a random variable $N$, 
then 
\begin{equation*}
    \lim_{n\rightarrow\infty}\mathbb{E}\exp\Big[\iota t\frac{S_{N_n}}
    {\sigma_n}\Big]
    =\int_{0}^{\infty}e^{-t^2x/2}\,\mu(dx), 
    \ t\in\mathbb{R},
\end{equation*}
where $\mu$ is the probability distribution of $1+N\xi$.
\end{theorem}
\begin{remark}
     Observe that 
     the condition $\sqrt{\mathbb{V}\mathrm{ar}[N_n]}=o(\sigma_n^2)$
     in Theorem \ref{prop:classical} (i) is redundant, 
     because 
    $\sigma_n^2=\mathbb{E}[N_n]\mathbb{V}\mathrm{ar}[Y_1]+\mathbb{V}\mathrm{ar}[N_n](\mathbb{E}[Y_1])^2$ implies,
$\mathbb{V}\mathrm{ar}[N_n]=O(\sigma_n^2)$.
\end{remark}
A non-commutative probability space (NCP)
is an extension of a classical probability space to a non-commutative setting. Our primary goal is to study random sums of variables (as opposed to random variables in a probability space) from an NCP, in particular, to establish results similar to those above and their consequences.

Free probability counterparts of various classical results for the deterministic sum of random variables are known. For example, the free central limit theorem (see \cite{Pata1996}) and the law of large numbers (see \cite{Bercovici1996}).  The free additive convolutions of probability measures on $\mathbb{R}$
were studied in \cite{Voiculescu1986}, \cite{Maassen1992} and \cite{Bercovici1993}, respectively, for measures that are compactly supported, 
that have finite variances, and that are arbitrary. 
In \cite{Anshelevich2000}, weak approximations of some free probability measures using sums of self-adjoint operators were obtained. In \cite{Bercovici1999}, various characterizations and limiting distributional properties of sums of identically distributed free variables were studied. An explicit analytical correspondence between the limit laws of free additive convolution and those of classical additive convolution was established. A generalization of these results to non-identically distributed free variables was obtained in \cite{Chistyakov2008}. For more details on different types of convergence for sums of free variables, see \cite{Bercovici1992, Bercovici1995, Kargin2007}, and references therein.

A special case of random sums in an NCP was studied in \cite{Bose2026}. Let $\{A_{j,n}\}$ be suitable independent $n\times n$ random patterned matrices, considered as elements of the matrix algebra with the state as the average trace. Let $N_n$ be a Poisson random variable which is independent of the matrices $\{A_{j,n}\}$.
They studied the  $*$-convergence of $n^{-1}(A_{1,n}+\dots+A_{N_n,n})$, and the weak limit of the empirical distributions of its eigenvalues. 

Let 
$(\mathcal{A},\varphi)$ be a $*$-probability space or an NCP, with the unity $1_\mathcal{A}$. Let $(\Omega,\mathcal{F},\mathbb{P})$ be a probability space. Let $(\mathcal{G}, \mathbb{E})$ be the (commutative) NCP, where $\mathcal{G}$ includes all non-negative random variables on $(\Omega, \mathcal{A},\mathbb{P})$ with all moments finite, and $\mathbb{E}$ is the expectation state. 
Let $N\in \mathcal{G}$ which is non-negative and integer valued, and $\{a_i\}$ be variables from $\mathcal{A}$, where $a_0=0$. We shall be interested in quantities such as $a_N$, $a_1+\cdots +a_N$, and $(a_1+\cdots +a_N)/\sqrt{N}$. By convention, empty sums are taken to be $0$. Empty products and $0/0$ are taken to be $1_\mathcal{A}$.  
Let $\tilde{\mathcal{A}}$ be the unital $*$-algebra  
generated by variables of the above type.
Define 
\begin{equation}\label{sumdef}
    \mathcal{S}_N\coloneqq a_1+\dots+a_N.
\end{equation}
In particular, $N\cdot1_\mathcal{A}\in \tilde{\mathcal{A}}$, 
and $1_\mathcal{A}$ continues to be the unity of this $*$-algebra. 
Define the linear functional $\tilde{\varphi}\coloneqq\mathbb{E}\varphi:\tilde{\mathcal{A}}\rightarrow\mathbb{C}$. Then, $(\tilde{\mathcal{A}},\tilde{\varphi})$ is a $*$-probability space or an NCP as the case may be. Let $\tilde\kappa$ denote the corresponding free cumulant function. 

Let $N_n\in \mathcal{G}$ be non-negative integer valued random variables 
and $\{a_j\}\subset \mathcal{A}$ be free, independent or half independent.
Define the standardized variable \begin{equation}\label{SNtilde}
\tilde{\mathcal{S}}_{N_n}\coloneqq\frac{\mathcal{S}_{N_n}-\tilde{\varphi}(\mathcal{S}_{N_n})}{\sqrt{\tilde{\kappa}_2(\mathcal{S}_{N_n})}}.
\end{equation} 
Relation (\ref{introre1}) is in terms of characteristic functions. However, in an NCP, the primary mode of convergence is that of ``moments'' computed with respect to the state. We establish approximations, inter alia, convergence results for moments 
of $(\tilde{\mathcal{S}}_{N_n}, \tilde{N_{n}})$
under appropriate conditions on the moments of $N_n$ and the $a_j$'s. Our method uses combinatorial relations between moments and cumulants/free cumulants/half cumulants. When the moments of the variables involved define suitable probability measures, then, these results lead to weak convergence results for the corresponding probability measures. In (\ref{introre1}), the Gaussian law arose. In our case, the semi-circle law arises via a semi-circle variable.  

In Section \ref{pre}, we recall some known concepts and results from non-commutative probability.
In Section \ref{sec3}, we outline the construction of the required NCP $(\tilde{\mathcal{A}},\tilde{\varphi})$,
and make some observations about the corresponding state and free cumulants.

In Section \ref{sec4}, we start with an extension of Theorem \ref{prop:classical}. 
We then state our main result in Theorem \ref{thm31}, which is the free analogue of that. 
A consequence is the following:   
suppose $\{a_i\}$ in (\ref{sumdef}) are free independent, and satisfy $\varphi(a_i)=\alpha\neq0$, $\kappa_2(a_i)=\theta^2<\infty$ and $\sup_{i}|\varphi(a_i^m)|<\infty$ for each $m\in\mathbb{N}$. Then under some appropriate moments conditions on $N_n$, the joint $*$-convergence of $\tilde{\mathcal{S}}_{N_n}$ 
and $\delta_n\tilde{N}_n=\delta_n\frac{N_n-\mathbb{E}[N_n]}{\sqrt{\mathbb{V}\mathrm{ar}[N_n]}}$ is equivalent to the convergence of $\delta_n\coloneqq\frac{\alpha\sqrt{\mathbb{V}\mathrm{ar}[N_n]}}{\sqrt{\theta^2\mathbb{E}[N_n]+\alpha^2\mathbb{V}\mathrm{ar}[N_n]}}$ and convergence of all moments of $\delta_n\tilde{N}_n$. Moreover, when $\varphi(a_i)=0$, and the sequence $\{(\frac{N_n}{\mathbb{E}[N_n]})^m\}$ is uniformly integrable for each $m\in\mathbb{N}$, their joint convergence is equivalent to the convergence of moments of $\frac{N_n}{\mathbb{E}[N_{n}]}$. A clt
for  $\tilde{\mathcal{S}}_{N_n}$ also follows.
In Section \ref{proof}, we provide some explicit moment calculation of $\tilde{\mathcal{S}}_{N_n}$, and the proof of 
Theorem \ref{thm31}.

In Section \ref{sec5}, we establish analogues of the results from Section \ref{sec4}, 
when 
$\{a_i\}$ 
are independent or half-independent. In both settings, the approximations and limiting behavior are similar to those in the free case. However, the approximations and limits involve standard Gaussian and standard symmetrized Rayleigh variables, respectively, instead of the semi-circular variable that arises when $\{a_i\}$ are free.

In Section \ref{sec6}, we investigate the $*$-convergence of randomly indexed variables. Let $\{a_n\}$  be self-adjoint variables.
Let $\{K_n\}$ be positive  numbers such that $(a_n-\alpha)/K_n \stackrel{*}{\to}l$ for some $\alpha\in\mathbb{R}$, as $n\rightarrow\infty$. Let $\{N_n\}$ be non-negative integer valued random variables. Suppose for a positive sequence $\{w_n\}$ increasing to $\infty$, $N_n/w_n\stackrel{\mathbb{P}}{\to} K$, where $K>0$ is a constant. Then, under an appropriate condition on $\{a_i\}$, 
we show that $(a_{N_{n}}-\alpha)/K_{w_n}\stackrel{*}{\to}l$.  
In Section \ref{sec7}, we derive clt-type results for random sums of self-adjoint variables with random scaling.
 In Section \ref{sec8}, we discuss some examples.
 
\section{Preliminaries}  \label{pre}
The sets of real, complex and natural numbers will be denoted by $\mathbb{R}$, $\mathbb{C}$ and $\mathbb{N}$, respectively. Convergence in probability will be denoted by $\stackrel{\mathbb{P}}{\to}$.
We cover some basic notions of non-commutative probability that are required for this article. For further details, see \cite{Nica2007, Voiculescu1992}.
\vskip5pt

\noindent \textbf{{Non-commutative probability and $*$-convergence.}}
 Let $\mathcal{A}$ be a unital algebra with unit element $1_\mathcal{A}$, and let $\varphi:\mathcal{A}\rightarrow\mathbb{C}$ be a linear functional with $\varphi(1_{\mathcal{A}})=1$. Elements of $\mathcal{A}$ will be called variables. The pair $(\mathcal{A},\varphi)$ is called a \textit{non-commutative probability} (NCP) space, and $\varphi$ is called a \textit{state}. It is called \textit{tracial} if $\varphi(ab)=\varphi(ba)$ for all $a,b\in\mathcal{A}$. If $\mathcal{A}$ is a $*$-algebra and $\varphi$ is \textit{positive}, that is, $\varphi(aa^*)\ge0$ for all $a\in\mathcal{A}$, then $(\mathcal{A},\varphi)$ is called a \textit{$*$-probability space}. 

Let $(\mathcal{A},\varphi)$ be an NCP. Unital sub-algebras $\{\mathcal{A}_i\}_{i\in I}$ of $\mathcal{A}$ are called \textit{independent} if for $r=1,2,\dots,k$, $k\in\mathbb{N}$, for any finite subset $J\subseteq I$, and $a_j\in\mathcal{A}_j$ for $j\in J$, $\varphi(\prod_{j\in J}a_j)=\prod_{j\in J}\varphi(a_j)$. They are called \textit{free independent} if for $a_{r}\in\mathcal{A}_{j_r}$, $j_r\in J$ such that $j_{l}\ne j_{l+1}$ for all $l=1,2,\dots,k-1$ and $\varphi(a_r)=0$ for all $r$, we have $\varphi(a_1\dots a_n)=0$. Variables $\{a_i\}_{i\in I}\subset\mathcal{A}$ are called \textit{freely independent (free)} if the sub-algebras generated by them are free. This concept can be extended to $*$-freeness in a $*$-probability space by including the starred variables. If an algebra or $*$-algebra is generated by free variables, then the state on this algebra is tracial. 
\vskip5pt

Let $(\mathcal{A},\varphi)$ be an NCP. The $m$th moment of any self-adjoint variable $a\in \mathcal{A}$  is defined as $\varphi(a^m)$ for $m\in\mathbb{N}$. If there is a unique probability measure $\mu_a$ with moments $\varphi(a^m)$, then $\mu_a$ is called the probability measure of $a$. For $m\in\mathbb{N}$, the collection $\{\varphi(a_{i_1}a_{i_2}\dots a_{i_k}),\ i_1,\dots,i_k\in\{1,\dots,m\},\ k\in\mathbb{N}\}$ is called 
the joint moments of $\{a_1,\dots,a_m\}\subset\mathcal{A}$.
Let $\mathbb{C}\langle x_1,\dots,x_m\rangle$ denote the algebra of polynomials in non-commutative indeterminates $x_1,\dots, x_m$. The joint distribution 
of $a_1,\dots,a_m$ is a linear functional $\mu_{a_1,\dots,a_m}:\mathbb{C}\langle x_1,\dots,x_m\rangle\rightarrow\mathbb{C}$ given by:
\begin{equation*}
\mu_{a_1,\dots,a_m}(Q)=\varphi(Q(a_1,\dots,a_m)).
\end{equation*}
In the case of a $*$-probability space, the joint distribution is extended by including the starred variables in the collection of moments ($*$-moments), and allowing $*$-variables in $Q$.
Let $(\mathcal{A}_n,\varphi_n)$, $n\ge1$ be NCPs. The variables 
$\{a_{1,n},\dots,a_{m,n}\}\subset\mathcal{A}_n$ are said to (jointly) converge  if
$\lim_{n\rightarrow\infty}\varphi_n(Q(a_{1,n},\dots,a_{m,n}))
$ exists and is finite for all $Q\in\mathbb{C}\langle x_1,\dots,x_m\rangle.$
To verify the above condition, it is sufficient to check it for monomials only. Moreover, the limits canonically define a
limit NCP $(\mathcal{A}, \varphi)$, where $\mathcal{A}$ is generated by indeterminates $a_1, \ldots, a_m$ and $\varphi(Q(a_1,\dots,a_m))$ is defined as the above limit. We say 
$(a_{1,n},\dots,a_{m,n})$ converges to $(a_1, \ldots, a_m)$. If $a_1,\dots,a_m$ are free, then $a_{1,n},\dots, a_{m,n}$ are said to be asymptotically free. All these notions can be extended in an obvious way  
when we have $*$-probability spaces. We shall express $*$-convergence as $(a_{1,n},\dots,a_{m,n})\stackrel{*}{\to} (a_1, \ldots, a_m)$. 
\vskip5pt

\noindent\textbf{{Moments, cumulants and free cumulants}.} 
To establish $*$-convergence, it is often convenient to use relations between moments, cumulants, and free cumulants.  For $m\in\mathbb{N}$, let $\mathcal{P}(m)$ be the set of all \textit{partitions} of $\{1,2,\dots,m\}$ with the usual partial order $\leq$ where $\textbf{1}_n=\{1,\dots,m\}$ and $\textbf{0}_n=\{\{1\},\{2\},\dots,\{m\}\}$ are the largest and the smallest partitions, respectively.  For any $\pi=\{V_1,\dots,V_r\}\in\mathcal{P}(m)$, $1\leq r\leq m$, let 
 \[|\pi|:=r=\#\ \text{of blocks of}\ \ \pi, \ \text{and}\ \ |V_j| :=\# \ \text{number of elements in}\ \  V_j.
 \]
Let  $\{c_k\}_{k\ge1}$ be a sequence of complex numbers. Its multiplicative extension $\{c_\pi\}, \pi \in \mathcal{P}_m, m\geq 1$ is defined as
 \[c_\pi :=c_{|V_1|}\dots c_{|V_r|}, \ \text{when}\ \ \pi=\{V_1,\dots,V_r\}\in\mathcal{P}(m).
 \]
Let $(\Omega,\mathcal{F},\mathbb{P})$ be a probability space and  
let $Y$ be any real random variable on this space with all moments $M_k=\mathbb{E}Y^k<\infty$. Then there exists a unique sequence $c_k$, $k\ge1$, called the cumulant sequence, that satisfies the relations 
\begin{equation}\label{eq:classicalmoment}
    M_k=\sum_{\sigma\in\mathcal{P}(k)}c_\sigma, \text{and}\ \ 
    M_\pi=\sum_{\sigma\in\mathcal{P}(m):\sigma\leq \pi}c_\sigma\ \text{for all}\ \pi\in\mathcal{P}(k),\ m, k\ge1,
\end{equation}
where $\{c_\sigma\}$ and $\{M_{\pi}\}$ are the multiplicative extension of $\{c_k\}$ and $\{M_{k}\}$. The reverse relations can be formulated, but we shall not need them. It may be noted that 
(\ref{eq:classicalmoment}) also holds for variables in an NCP, and its appropriate extension holds for product moments of commuting variables in an NCP, where the classical joint moments are replaced by the joint $*$-moments of variables. 

Let $(\mathcal{A},\varphi)$ be an NCP. Then define the moment
functionals on $\mathcal{A}^k$ as:
\[	\varphi_{k}(a_1,\dots,a_k)=\varphi(a_1\cdots a_k), \ k \geq 1.
\]
The functionals $\{\varphi_k\}$ are multilinear. Their multiplicative extension is defined as follows.
Let $\pi=\{V_1,\dots,V_r\}\in \mathcal{P}(k)$, $1\leq r\leq k$,
where $V_j=\{i_{j1}, \ldots i_{j|V_{j}|}\}$ (elements written in ascending order), $ 1\leq j \leq r$ . Then 
\begin{equation}\label{mommultext}
	\varphi_{\pi}(a_1,\dots,a_k)=\varphi_{|V_{1}|}(a_{i_{11}},\dots,a_{i_{1|V_{1}|}})\cdots\varphi_{|V_r|}(a_{i_{r1}},\dots,a_{i_{r|V_r|}}).
\end{equation} 
Let $NC(k)\subset \mathcal{P}(k), \ k \geq 1$ be the set of  \textit{non-crossing partitions}. Free cumulant functionals are defined using only $\{\varphi_{\pi}, \pi \in NC(k), k \geq 1\}$. Since $NC(k)$ is a finite lattice, it has a M{\"o}bius function, denoted by a common notation $\mathrm{Mob}[\cdot,\cdot]$ for all $k$. For $k\in\mathbb{N}$, define the free cumulant functionals 
via $\mathrm{Mob}$ and $\{\varphi_{\pi}\}$ as:
\begin{equation}\label{cummomrelation}
	\kappa_k(a_1,\dots,a_k)\coloneqq\sum_{\pi\in NC(k)}\mathrm{Mob}[\pi,\textbf{1}_k]\ \varphi_{\pi}(a_1,\dots,a_k),\ k\in\mathbb{N}.
\end{equation}
For $a\in \mathcal{A}$, we write $\kappa_k(a)$ for $\kappa_k(a, \ldots, a)$.
Multiplicative extension $\{\kappa_{\pi}\}$
of $\{\kappa_k\}$ is defined similarly, and they satisfy the relation:
\begin{equation*}
    \kappa_\pi(a_1,\dots,a_k)=\sum_{\sigma\in NC(k): \sigma\leq \pi}\text{Mob}[\sigma,\pi]\varphi_\sigma(a_1,\dots,a_k), \ \pi\in NC(k), \ k \geq 1.
\end{equation*}
The inversion of the M{\"o}bius function gives
\begin{equation}\label{momcumrelation}
    \varphi_k(a_1,\dots,a_k)\coloneqq\varphi(a_1\dots a_k)=\sum_{\pi\in NC(k)}\kappa_\pi(a_1,\dots,a_k), \ k \geq 1,
\end{equation}
and more generally, 
\begin{equation}\label{momcumrelationext}
    \varphi_{\pi}(a_1,\dots,a_k)=\sum_{\sigma\in NC(k),\ \sigma \leq \pi }\kappa_\sigma(a_1,\dots,a_k), \ \pi \in NC(k), \ k \geq 1.
\end{equation}
It is known that self-adjoint  $\{a_i\}_{i\in I}\subset\mathcal{A}$ are free iff all their mixed free cumulants are $0$. That is, for all $k\geq 2$ and $i_1,\dots,i_k\in I$ such that at least two of these indices are distinct, we have
\begin{equation}\label{jfcum=0}
    \kappa_k(a_{i_1},\dots,a_{i_k})=0.
\end{equation} 
As a consequence, if $a_1,\dots,a_m$ are free then 
\begin{align}
    \varphi(a_1\cdots a_m)&=\varphi(a_1)\cdots\varphi(a_m)\label{fmomformula},\\ \kappa_k(a_1+\dots+a_m)&=\kappa_k(a_1)+\cdots+\kappa_k(a_m), \ k \geq 1.\label{fcumformula}
\end{align}
Natural extensions of these notions are available for non-self adjoint variables and for $*$-probability spaces. 
\vskip5pt

Let $(\mathcal{A},\varphi)$ be a $*$-probability space. A self-adjoint variable $z\in\mathcal{A}$ is called a standard Gaussian variable if $\mu_z$ is the standard Gaussian probability measure.
A self-adjoint variable $s\in\mathcal{A}$ is called a standard semi-circle variable if 
$\mu_s$ has the density 
\begin{equation*}
    f_s(x)=\frac{1}{2\pi}\sqrt{4-x^2},\ x\in[-2,2].
\end{equation*}
A self-adjoint variable $r_S\in\mathcal{A}$ is called a standard  symmetrized Rayleigh variable if 
$\mu_{r_S}$ 
has the density
\begin{equation*}
    f_{r_S}(x)=|x|e^{-x^2},\ x\in\mathbb{R}.
\end{equation*}

\vskip5pt

\noindent \textbf{CLT in NCP}. We shall need the following three clts, which are usually proved with the additional condition that the summands are identically distributed. Proofs for the slightly more general statement use (\ref{gholder}). 
 Part (i) is the well known free clt (see \cite{Nica2007}). Part (ii) is the moment version of the classical clt, and can be proved easily using the moment--cumulant relations. Part (iii) is proved in \cite{Bose2011}, and uses half-cumulants. 

\begin{theorem}\label{thm:generalclt} Suppose $(\mathcal{A}, \varphi)$ is an NCP. 
Let $\{a_i\}$ be self-adjoint variables such that, $\varphi(a_i)=0, \varphi(a_i^2)=1$, and  $\sup_i|\varphi(a_i^m)| < \infty$ for every positive integer $m$. 
Then the following hold for $S_n:=a_1+\cdots +a_n$.\vskip5pt

\noindent (i) If $\{a_i\}$ are free, then $S_n/\sqrt{n}\stackrel{*}{\to} s$.
\vskip5pt

\noindent (ii) If $\{a_i\}$ are independent, then $S_n/\sqrt{n}\stackrel{*}{\to} z$.
\vskip5pt

\noindent (iii) If $\{a_i\}$ are half-independent, then $S_n/\sqrt{n}
\stackrel{*}{\to} r_S$.
\vskip3pt

\noindent In other words, $\varphi(S_n/\sqrt{n})^m$ in (i), (ii) and (iii) respectively, equal $\varphi(s^m)+o(1)$, $\varphi(z^m)+o(1)$, $\varphi(r_S^m)+o(1)$.
\end{theorem}
We will use the following lemma, 
to convert $*$-convergence results to weak convergence results. 

\begin{lemma}\label{lem:momentconv} Let $\{X_n\}$ be real valued random variables with probability measures $\{\mu_n\}$. Suppose for every $m=0, 1, \ldots$, 
$\mathbb{E}(X_n^m)$ converges to say $\beta_m$ (finite). Then there is
a probability measure $\mu$ whose moments are $\{\beta_m\}$. If $\{\beta_m\}$ determines a unique probability measure
$\mu$, then $\mu_n$ converges weakly to $\mu$. 
\end{lemma}

\section{NCP of random sums and products}\label{sec3}
Let $(\Omega,\mathcal{F},\mathbb{P})$ be a probability space. 
Then the collection of all 
discrete random variables taking values in $\mathbb{N}\cup\{0\}=\{0, 1, 2, \ldots\}$ whose all moments are finite 
is a $*$-probability space with the expectation $\mathbb{E}$ as the state. We write this space as $(\mathcal{G}, \mathbb{E})$. Let $(\mathcal{A},\varphi)$ be any other NCP (or $*$-probability space) which is (tensor) independent of $(\mathcal{G}, \mathbb{E})$. 

We now define a new NCP by combining the above two NCPs in the following ``independent'' way. Consider all random sums and products of elements from $\mathcal{A}$ defined in the natural way: for example, if $\{a_i, i \geq 1\}$ are elements from $\mathcal{A}$, and $N\in \mathcal{G}$
then,
\begin{equation}\label{sumprodalg}
    a_1+\cdots +a_N=\sum_{i=0}^\infty (a_1+\cdots +a_i)\mathbb{I}_{\{N=i\}}, \ \  a_1\cdots a_N=\sum_{i=0}^\infty (a_1 \cdots  a_i)\mathbb{I}_{\{N=i\}}.
\end{equation}
The empty sum and product in (\ref{sumprodalg}) are equal to $0$ and $1_\mathcal{A}$, respectively. 
The collection of sums and products of the above quantities
forms an algebra, and we denote it by $\tilde{\mathcal{A}}$. Note that quantities such as $(a_1+\cdots +a_N)/\sqrt{N}$ are well defined (we take $0/0=1$).  
Let us consider the following ``random linear functional (state)'' 
(which we continue to denote by $\varphi$), 
defined in the natural way:
\[\varphi(a_1+\cdots + a_N)=\sum_{i=0}^\infty \varphi(a_1+\cdots +a_i)\mathbb{I}_{\{N=i\}}.\] 
We continue to denote the corresponding ``random'' free cumulant by $\kappa$, and $\mathbb{E}\kappa$ will denote the expectation of $\kappa$ with respect to the distribution of the random variables involved.

Taking an expectation with respect to the random variables involved, we define a state  $\tilde{\varphi}\coloneqq\mathbb{E}\varphi$ on $\tilde{\mathcal{A}}$. The corresponding free cumulant will be denoted by $\tilde\kappa$. 
Clearly,
\begin{align}
    \tilde{\varphi}(a_1+\cdots+a_N)&\coloneqq\sum_{i=0}^{\infty}\mathbb{P}\{N=i\}\varphi(a_1+\cdots+a_i),\label{phitildes}\\
    \tilde{\varphi}(a_1\cdots a_N)&\coloneqq\sum_{i=0}^{\infty}\mathbb{P}\{N=i\}\varphi(a_1\cdots a_i).
\end{align}

We make the following observations about $\tilde\kappa$ and $\mathbb{E}\kappa$.

\begin{remark}
    \noindent (i) It turns out that 
 $\tilde{\kappa}\neq \mathbb{E}\kappa$. To see this, let $\tilde{a}_N$ be an element in the NCP $(\tilde{\mathcal{A}},\tilde{\varphi})$ of type either $a_1+\cdots+a_N$ or $a_1\cdots a_N$. Its $m$th free cumulant $\tilde\kappa$ is given by
\begin{align*}
    \tilde{\kappa}_m(\tilde{a}_N)=\sum_{\pi\in NC(m)}\text{Mob}[\pi,1_m]\prod_{V\in\pi}\tilde{\varphi}_{|V|}(\tilde{a}_N)=\sum_{\pi\in NC(m)}\text{Mob}[\pi,1_m]\prod_{V\in\pi}\mathbb{E}\varphi_{|V|}(\tilde{a}_N).
\end{align*}
As $\prod_{V\in\pi}\mathbb{E}\varphi_{|V|}(\tilde{a}_N)$ is in general not equal to $\mathbb{E}\prod_{V\in\pi}\varphi_{|V|}(\tilde{a}_N)$, it follows that $\tilde{\kappa}\neq \mathbb{E}\kappa$. \vskip5pt 

\noindent (ii) Let $a_1,a_2,\ldots\in\mathcal{A}$ be identically distributed free variables. Let us consider the following partial random sum:
$\mathcal{S}_N\coloneqq a_1+\dots+a_{N}$. 
Using (\ref{momcumrelation}) and (\ref{phitildes}), we have for every $m \geq 1$, 
    \begin{align}
\tilde{\varphi}(\mathcal{S}_N^m)&=\sum_{\pi\in NC(m)}\tilde{\kappa}_\pi(\mathcal{S}_N) 
        \nonumber\\ 
        &=\sum_{i=1}^{\infty}\mathbb{P}\{N=i\}\varphi[(a_1+\dots+a_i)^m]\label{rem11}\\
        &=\sum_{i=1}^{\infty}\mathbb{P}\{N=i\}\sum_{\pi\in NC(m)}\kappa_\pi(a_1+\dots+a_i)\nonumber\\
        &=\sum_{\pi\in NC(m)}\mathbb{E}\kappa_\pi(\mathcal{S}_N)\nonumber\\
    &=\sum_{\pi\in NC(m)}\mathbb{E} [N^{|\pi|}\kappa_\pi(a_1)],\nonumber
    \end{align}
    where the last equality holds since  $a_j$'s are free and identically distributed.
\end{remark}
 
\section{Random sums with non-random scaling}\label{sec4}
The following result is a joint approximation extension of 
Theorem \ref{prop:classical}. 
We omit its proof.
\begin{theorem} \label{thmjtcov} 
Let $\{N_n\}$ be non-negative integer-valued rvs
independent of iid rvs $\{Y_n\}$ with $0<\mathbb{E}[Y_n^2]<\infty$. Let $Z$ be a standard Gaussian rv. \vskip5pt

\noindent (i) Suppose $\mathbb{E}[Y_1]\neq0$, and 
as $n\rightarrow\infty$, $\sigma_n^2\rightarrow\infty$, then for $t_1,t_2\in\mathbb{R}$, we have 
\begin{equation*}
    \mathbb{E}\exp\Big[\iota t_1(\frac{S_{N_n}-\mathbb{E}[N_nY_1]}{\sigma_n})
    +\iota \delta_nt_2\tilde{N}_n\Big]=\mathbb{E}\exp[\iota\delta_n(t_1+t_2)\tilde{N}_n]\mathbb{E}\exp[\iota t_1\sqrt{1-\delta_n^2}Z]+o(1).
\end{equation*}
\noindent (ii) Suppose $\mathbb{E}[Y_1]=0$. 
If as $n\rightarrow\infty$, $\mathbb{E}[N_n]\rightarrow\infty$, and the probability measures of $\frac{N_n}{\mathbb{E}[N_n]}$ weakly converge to a probability measure $\mu$, then
\begin{equation*}
    \lim_{n\rightarrow\infty}\mathbb{E}\exp\Big[\iota t_1\frac{S_{N_n}}{\sigma_n}
    +\iota t_2\frac{N_n}{\mathbb{E}[N_n]}\Big]=\int_{0}^{\infty}\exp(-\frac{t_1^2}{2}x+\iota t_2x)\,\mu(\mathrm{d}x),\ t_1,t_2\in\mathbb{R}.
\end{equation*}
\end{theorem}

\begin{remark}
    Under conditions of Theorem \ref{thmjtcov}(i), weak convergence of 
$(\frac{S_{N_n}-\mathbb{E}[N_nY_1]}{\sigma_n}, \delta_n \tilde{N}_n)$ is equivalent to the convergence of $\delta_n$ and weak convergence of $\delta_n\tilde{N}_n$.
\end{remark}

We retain some notation from Theorem \ref{prop:classical} for the non-commutative case. Also, from now on, all rvs that appear are assumed to have finite moments of all orders.  
Recall the NCP $(\tilde{\mathcal{A}},\tilde{\varphi})$, defined in Section \ref{sec3}. Let $N_n$, $n\ge1$ be a non-negative integer valued rv with 
\begin{equation}\label{momsN}
    \eta_n :=\mathbb{E}N_n,\ \ \rho^2_n : =\mathbb{E}N_n^2\ \ \text{and}\ \ \gamma^2_n :=\mathbb{V}\mathrm{ar}(N_n)=\rho^2_n-\eta^2_n.
\end{equation} 
Let $a_1,a_2,\dots$ be free self-adjoint and identically distributed variables  with 
\begin{equation*}
    \alpha:=\varphi(a_1),\ \ \beta^2:=\varphi(a_1^2),\  \theta^2:=\kappa_2(a_1)=\beta^2-\alpha^2.
\end{equation*}
Let $\mathcal{S}_{N_n}\coloneqq a_1+\dots+a_{N_n}$ with $\mathcal{S}_0=0$. Then,
\begin{align}
    \tilde{\varphi}(\mathcal{S}_{N_n})&=\mathbb{E}\varphi(\mathcal{S}_{N_n})=\mathbb{E}N_n\varphi(a_1)=\eta_n\alpha,\label{fmomSN}\\
    \tilde{\varphi}(\mathcal{S}_{N_n}^2)&=\sum_{i=0}^{\infty}\mathbb{P}\{N_n=i\}\varphi[(a_1+\dots+a_i)^2]\nonumber\\
    &=\sum_{i=0}^{\infty}\mathbb{P}\{N_n=i\}\big[\sum_{j=1}^{i}\varphi(a_j^2)+\sum_{1\leq j\neq k\leq i}\varphi(a_j)\varphi(a_k)\big]\ \ \text{(using freeness)}\nonumber\\
    &=\sum_{i=0}^{\infty}\mathbb{P}\{N_n=i\}[i\beta^2+(i^2-i)\alpha^2]=\eta_n\theta^2+\rho_n^2\alpha^2,\label{smomSN}\\
    \sigma^2_n&:=\tilde{\kappa}_2(\mathcal{S}_{N_n})=\tilde{\varphi}(\mathcal{S}_{N_n}^2)-[\tilde{\varphi}(\mathcal{S}_{N_n})]^2=\eta_n\theta^2+\gamma_n^2\alpha^2.\label{tmomSN}
\end{align}
Define the following standardized variables:
\begin{equation*}
    \tilde{N}_n\coloneqq\frac{N_n-\eta_n}{\gamma_n}\ \ \text{and}\ \ \tilde{\mathcal{S}}_{N_n}\coloneqq \frac{\mathcal{S}_{N_n}-\eta_n\alpha}{\sigma_n}.
\end{equation*}
Let 
    \begin{equation}\label{deltadef}
        \delta_n\coloneqq\frac{\gamma_n \alpha}{\sigma_n}=\bigg(\frac{\gamma_n^2\alpha^2}{\eta_n\theta^2+\gamma_n^2\alpha^2}\bigg)^{1/2},\ 0\leq \delta_n\leq 1. 
    \end{equation}
    
Theorem \ref{thm31} is a free version of Theorem  \ref{thmjtcov}. Its proof is given in Section \ref{proof}. 

\begin{theorem}\label{thm31}
Let $\{a_i\}$ be free identically distributed self-adjoint variables in $(\mathcal{A}, \varphi)$.
Let $N_n$, $n\ge1$ be a rv, independent of $\{a_i\}$, and $s$ be a standard semi-circle variable, independent of $\tilde{N}_n\cdot 1_\mathcal{A}$.
\vskip5pt
\noindent    (i) Let $\varphi(a_1)\neq0$. Suppose as $n\rightarrow\infty$, 
    \begin{equation}\label{thm31cond1}
       \eta_n\rightarrow\infty,\ \sigma^2_n\rightarrow\infty\ \ \text{and} \ \ \mathbb{E}[|N_n-\eta_n|^m]=O(\sigma_n^{m})\  \ \text{for all}\ m\in\mathbb{N}.
    \end{equation} 
    Then,
    \begin{equation}\label{thm31re}
        \tilde{\varphi}(\tilde{\mathcal{S}}_{N_n}^m(\delta_n\tilde{N}_n)^l)=\tilde{\varphi}[(\delta_n\tilde{N}_n\cdot1_\mathcal{A}+\sqrt{1-\delta_n^2}s)^m(\delta_n\tilde{N}_n)^l]+o(1)\ \ \text{for all}\ \ m,l\in\mathbb{N}\cup\{0\}.
    \end{equation}
\noindent (ii)      
     Let $\varphi(a_1)=0$. Suppose as $n\rightarrow\infty$, $\eta_n\rightarrow\infty$ and
     \begin{equation}\label{thm31cond2}
         \limsup_{n\rightarrow\infty}\mathbb{E}\Big[(\frac{N_n}{\eta_n})^m\Big]<\infty\ \ \text{for all}\ m\in\mathbb{N}.
     \end{equation}
     Then, for each $m,l\in\mathbb{N}\cup\{0\}$, we have
     \begin{equation*}
         \tilde{\varphi}\Big[\tilde{S}_{N_n}^m(\frac{N_n}{\eta_n})^l\Big]=\mathbb{E}\Big[(\frac{N_n}{\eta_n})^{l+\frac{m}{2}}\Big]\varphi(s^m)+o(1).
     \end{equation*}
\end{theorem}
\vskip5pt
Theorem \ref{thmjtcov} 
required only the second moment condition, $\mathbb{E}[(N_n-\eta_n)^2]=O(\sigma_n^2)$, enough to approximate characteristic functions.  
It seems natural that the joint $*$-moment approximation of $\{\tilde{\mathcal{S}}_{N_n},\tilde{N}_n\}$ in Theorem \ref{thm31} would require an appropriate all order moments condition, namely (\ref{thm31cond1}). 

\begin{remark}
Let $\alpha_n$ denote the probability distribution of $\delta_n\tilde{N}_n$. Suppose $a_1$ has a probability measure. Then for every $n$, $\sqrt{1-\delta_n^2}(a_1+\dots+a_{n}-\alpha n)/\sqrt{n}\theta$ has a probability measure, $\mu_n$ say. 
It then implies that 
$\tilde{\mathcal{S}}_{N_n}$ has a probability measure, $\nu_n$ say.
From Theorem \ref{thm31}, it follows that for large $n\rightarrow\infty$, $\nu_n$ is approximated by the standard additive convolution of $\mu_n$ and $\alpha_n$. 
\end{remark}

The following corollary is an immediate consequence of 
Theorem \ref{thm31}(i).
\begin{corollary}\label{corr41}
    \noindent Under conditions of Theorem \ref{thm31}(i), $*$-convergence of 
$\{\tilde{\mathcal{S}}_{N_n}, \delta_n\tilde{N}_n\cdot1_\mathcal{A}\}$ is equivalent to the convergence of $\delta_n$ and moments of $\delta_n\tilde{N}_n$. Further, if $\delta_n\to \delta\neq0$, then the $*$-convergence above is equivalent to the 
convergence of moments of $\tilde{N}_n$.
\end{corollary}

Our next corollary yields a free clt for the random sum.
\begin{corollary}\label{crr31}
    Suppose $\alpha\neq 0$, (\ref{thm31cond1}) holds, and as $n\rightarrow\infty$,
    \begin{equation}\label{pf21}
        \gamma_n^2=o(\eta_n).
    \end{equation}
    Then,
$\tilde{\mathcal{S}}_{N_n}\stackrel{*}{\to}s$. 
If $\tilde{\mathcal{S}}_{N_n}$ have probability measures $\mu_n$, then they converge weakly to $\mu_s$. 
\end{corollary}
\begin{proof}
    Using (\ref{pf21}) in (\ref{deltadef}), it follows that $\lim_{n\rightarrow\infty}\delta_n=0$. For the first part, in view of (\ref{thm31re}), it is enough to show that $\lim_{n\rightarrow\infty}\mathbb{E}[\delta_n\tilde{N}_n]^m=0$ for all $m\in\mathbb{N}$. Note that $\mathbb{E}[\delta_n\tilde{N}_n]=0$ and $\mathbb{V}\mathrm{ar}(\delta_n\tilde{N}_n)=\delta_n^2\rightarrow0$ as $n\rightarrow\infty$. Hence, 
    $\delta_n\tilde{N}_n\stackrel{\mathbb{P}}{\to} 0$. 
 Moreover, using the 
    third condition of (\ref{thm31cond1}), for a finite constant $C>0$, for each $m$, we get
    \begin{equation*}
        \mathbb{E}[|\delta_n\tilde{N}_n|^m]= 
        \frac{\delta_n^m}{\gamma_n^m} [\mathbb{E}|{N_n-\eta_n}|^{m}]=
        \frac{O(\sigma_n^m)}{\sigma_n^m}\alpha^m\leq C\alpha^m<\infty.
    \end{equation*}
    Thus, the sequence $\{(\delta_n\tilde{N}_n)^m\}$ is uniformly integrable for each $m\in\mathbb{N}$. 
    Hence, all its moments converge to $0$. This proves the first part. The second part now follows instantly from Lemma \ref{lem:momentconv}, since all moments of $\tilde{\mathcal{S}}_{N_n}$ converge to the moments of a semi-circle measure, which is compactly supported, and hence uniquely determined by its moments. 
\end{proof}

Now we consider the situation where 
$\tilde{N}_n$ converges weakly.
\begin{corollary}\label{corr32}
    Suppose $\alpha\neq0$, 
    $\tilde{N}_n$ converges weakly to $N$ with  
    probability measure $\mu_N$, 
    and $\{\tilde{N}_n^m\}$ is uniformly integrable for all $m\in\mathbb{N}$. 
    Then, as $n\rightarrow\infty$, 
    \begin{equation*}
        \tilde{\varphi}(\tilde{\mathcal{S}}_{N_n}^m)=\varphi[(\delta_na+\sqrt{1-\delta_n^2}s)^m]+o(1), \ \text{for each}\ \ m\in \mathbb{N}.
    \end{equation*}
    where $a$ is a self-adjoint variable whose all moments agree with those of $\mu_N$, and $s$ is a standard semi-circle variable, independent of $a$. 
 Further, if $\mu_a$ exists, then for every $n$, a probability law for  $\delta_na\cdot1_\mathcal{A}+\sqrt{1-\delta_n^2}s$ also exists. 
\end{corollary}

\begin{proof} 
By uniform integrability of all powers of $\tilde{N}_n$, all moments of $N$ are finite. In addition, using $\alpha^2\gamma_n^2\leq \sigma_n^2$, for all $k$,
    \begin{equation*}
        \mathbb{E}[|N_n-\eta_n|^k]=\gamma_n^k\mathbb{E}|N|^k+o(\gamma_n^k)\leq \frac{\sigma_n^k}{|\alpha|^k}[\mathbb{E}|N|^k+o(1)].
    \end{equation*}
     Thus, the third condition in (\ref{thm31cond1}) holds.
    This completes the proof of the first part. Now note that $\mu_s$ exists. Thus, if $\mu_a$ 
    exists, then the required limiting probability measure is the (weighted) additive convolution of $\mu_a$ and $\mu_s$. 
\end{proof}

\begin{corollary}\label{corr33}
    Suppose (\ref{thm31cond1}) holds, $\alpha\neq 0$, and 
    $\tilde{N}_n$ has a limiting 
     probability measure $\mu$. 
    If the following limit:
    \begin{equation}\label{crr33}
        \lim_{n\rightarrow\infty}\frac{\eta_n\theta^2}{\alpha^2\gamma_n^2}\coloneqq\zeta\in[0,\infty),
    \end{equation}
    exist, then $\tilde{\mathcal{S}}_{N_n}\stackrel{*}{\to}
    (1+\zeta)^{-1/2}a+\sqrt{\zeta/(1+\zeta)}s$, which is self-adjoint, and $s$ and $a$ are independent, 
    and $a$ is a self-adjoint variable whose $m$th moment is given as
    $\varphi(a^m)=\int_{\mathbb{R}}x^m\,\mathrm{d}\mu(\mathrm{d}x)$, $m\in\mathbb{N}$. 
    Moreover, if 
    $\tilde{\mathcal{S}}_{N_n}$ and $(1+\zeta)^{-1/2}a+\sqrt{\zeta/(1+\zeta)}s$ have probability measures $\nu_n$, and $\nu$, 
    then as $n \to \infty$, $\nu_n$ converges weakly to $\nu$.
    \end{corollary}
\begin{proof}
    From (\ref{crr33}), we get $\lim_{n\rightarrow\infty}\delta_n=(1+\zeta)^{-1/2}>0$. Then, the probability measure of $\delta_n\tilde{N}_n$ converges weakly to a valid probability measure. 
    Moreover, from the proof of Corollary \ref{crr31}, $\{(\delta_n\tilde{N}_n)^m\}$ is uniformly integrable for each $m\in\mathbb{N}$. Thus,
    \begin{equation*}
        \lim_{n\rightarrow\infty}\mathbb{E}[\delta_n\tilde{N}_n]^m=(1+\zeta)^{-m/2}\int_{\mathbb{R}}x^m\,\mu(\mathrm{d}x),\ m\in\mathbb{N}.
    \end{equation*}
So the first part follows from Theorem \ref{thm31}. The second part follows from Lemma \ref{lem:momentconv}. 
\end{proof}

\begin{corollary}\label{corr45}
    Suppose the assumptions of Theorem \ref{thm31}(ii) hold. 
    Suppose $\frac{N_n}{\eta_n}$ has a limiting probability measure $\mu$. Then, 
    \begin{equation*}
        \lim_{n\rightarrow\infty}\tilde{\varphi}(\tilde{\mathcal{S}}_{N_n}^m)=\varphi(s^m)\int_{0}^\infty x^{m/2}\,\mu(\mathrm{d}x),\ m\in\mathbb{N}.
        \end{equation*}
        In other words, $\tilde{\mathcal{S}}_{N_n}\stackrel{*}{\to}
        sy$ where $s$ and $y$ are independent variables, 
        and the $m$th moment of $y$ is given by $\int_0^\infty x^{m/2}\,\mu(\mathrm{d}x)$. If further $y$ has a probability measure $\mu_y$, and
        $\tilde{\mathcal{S}}_{N_n}$ has a probability measure $\mu_n$, then $\mu_n$ converges weakly to $\mu_s \star \mu_y$ where $\star$ denotes product convolution of two probability measures.   
    \end{corollary}
\begin{proof} 
In view of (\ref{thm31cond2}), the weak convergence of $N_n/\eta_n$ implies the convergence of all its moments. Now, the proof of the first part follows from Theorem \ref{thm31} (ii). The second part follows from Lemma \ref{lem:momentconv}.
\end{proof}

We end this section with some remarks on a multivariate extension of Theorem \ref{thm31}. Let $\{a_i^{(j)},\ 1\leq j\leq k\}_{i\ge1}$ be self-adjoint variables in an NCP $(\mathcal{A},\varphi)$, which are free across $i$. Let $N_n$, $n\ge1$ be a non-negative integer valued random variable, and set $\mathcal{S}_{N_n}^{(j)}\coloneqq a^{(j)}_1+\dots+a^{(j)}_{N_n}$, $j=1,\dots,k$. Then, under suitable extensions of conditions in Theorem \ref{thm31}, the standardized variables $\{\tilde{\mathcal{S}}_{N_n}^{(j)},\ 1\leq j\leq k\}$ jointly converge in the $*$-distribution. In view of the multivariate free central limit theorem (see Theorem 4.2.2 of \cite{Bose2022}), the limit will be the multivariate version of (\ref{thm31re}), except it will involve a semi-circle family.

\section{Proof of Theorem \ref{thm31}}\label{proof}
Before proving Theorem \ref{thm31}, we show explicit calculations for some values of $m$. It is assumed that $\alpha\neq0$. 
Write $s_{\delta_{n}}=(1-\delta_n^2)^{1/2}s$. For $m=1$, clearly, $\tilde{\varphi}(\tilde{\mathcal{S}}_{N_n})=0$ and $\tilde{\varphi}(\delta_n\tilde{N}_n\cdot1_\mathcal{A}+s_{\delta_n})=\delta_n\mathbb{E}\tilde{N}_n+\varphi(s_{\delta_n})=0$. For $m=2$, 
\begin{align*}
    \tilde{\varphi}(\tilde{\mathcal{S}}_{N_n}^2)&=\frac{1}{\sigma_n^2}[\tilde{\varphi}(\mathcal{S}_{N_n}^2)+\eta_n^2\alpha^2-2\eta_n\alpha\tilde{\varphi}(\mathcal{S}_{N_n})]\\
    &=\frac{1}{\sigma_n^2}[\eta_n\theta^2+\rho_n^2\alpha^2+\eta_n^2\alpha^2-2\eta_n^2\alpha^2]\\
    &=\frac{\eta_n\theta^2+\gamma_n^2\alpha^2}{\sigma_n^2}=1,
\end{align*}
where the second equality follows from (\ref{fmomSN}) and (\ref{smomSN}), and the last step follows from  (\ref{tmomSN}). 
Also,
\begin{equation*}
    \tilde{\varphi}[(\delta_n\tilde{N}_n\cdot1_\mathcal{A}+s_{\delta_n})^2]=\delta_n^2\mathbb{E}\tilde{N}_n^2+\varphi(s_{\delta_n}^2)+2\delta_n\mathbb{E}N_n\varphi(s_{\delta_n})=\delta_n^2+1-\delta_n^2=1.
\end{equation*}
Now, for $m=3$, using (\ref{fmomSN}) and (\ref{smomSN}), we get
\begin{align}
    \tilde{\varphi}(\tilde{\mathcal{S}}_{N_n}^3)&=\frac{1}{\sigma_n^3}[\tilde{\varphi}(\mathcal{S}_{N_n}^3)-\eta_n^3\alpha^3-3\eta_n\alpha\tilde{\varphi}(\mathcal{S}_{N_n}^2)+3\eta_n^2\alpha^2\tilde{\varphi}(\mathcal{S}_{N_n})]\nonumber\\
    &=\frac{1}{\sigma_n^3}[\tilde{\varphi}(\mathcal{S}_{N_n}^3)-\eta_n^3\alpha^3-3\eta_n^2\alpha\theta^2-3\eta_n\rho_n^2\alpha^3+3\eta_n^3\alpha^3]\nonumber\\
    &=\frac{1}{\sigma_n^3}[\tilde{\varphi}(\mathcal{S}_{N_n}^3)-3\eta_n^2\alpha\theta^2-3\eta_n\rho_n^2\alpha^3+2\eta_n^3\alpha^3],\label{Snttmom}
\end{align}
where 
\begin{align}
    \tilde{\varphi}(\mathcal{S}_{N_n}^3)=\mathbb{E}\varphi(\mathcal{S}_{N_n}^3)&=\mathbb{E}\sum_{\pi\in NC(3)}\kappa_\pi(\mathcal{S}_{N_n})\nonumber\\
    &=\mathbb{E}\kappa_3(\mathcal{S}_{N_n})+3\mathbb{E}\kappa_1(\mathcal{S}_{N_n})\kappa_2(\mathcal{S}_{N_n})+\mathbb{E}[\kappa_1(\mathcal{S}_{N_n})]^3\nonumber\\
    &=\mathbb{E}N_n\kappa_3(a_1)+3\mathbb{E}N_n^2\kappa_1(a_1)\kappa_2(a_1)+\mathbb{E}N_n^3[\kappa_1(a_1)]^3\nonumber\\
    &=\eta_n\kappa_3(a_1)+3\rho_n^2\alpha\theta^2+\mathbb{E}N_n^3\alpha^3.\label{SN3mom}
\end{align}
Substituting this in (\ref{Snttmom}) yields
\begin{align*}
    \tilde{\varphi}(\tilde{\mathcal{S}}_{N_n}^3)&=\frac{1}{\sigma_n^3}[\eta_n\kappa_3(a_1)+3\rho_n^2\alpha\theta^2+\mathbb{E}N_n^3\alpha^3-3\eta_n^2\alpha\theta^2-3\eta_n\rho_n^2\alpha^3+2\eta_n^3\alpha^3]\\
    &=\frac{1}{\sigma_n^3}[\eta_n\kappa_3(a_1)+3\gamma_n^2\alpha\theta^2+\mathbb{E}N_n^3\alpha^3-3\eta_n\rho_n^2\alpha^3+2\eta_n^3\alpha^3].
\end{align*}
From (\ref{tmomSN}), we have $\alpha^2\gamma_n^2\leq \sigma_n^2$ and $\theta^2\eta_n\leq \sigma_n^2$. So, $\gamma_n^2/\sigma_n^3=o(1)$ and  $\eta_n/\sigma_n^3=o(1)$, as $n\rightarrow\infty$. Thus, using $\sigma_n=\gamma_n\alpha/\delta_n$, as $n\rightarrow\infty$,
\begin{equation}\label{thm31pf1}
    \tilde{\varphi}(\tilde{\mathcal{S}}_{N_n}^3)=\frac{\delta_n^3}{\gamma_n^3}[\mathbb{E}N_n^3-3\eta_n\rho_n^2+2\eta_n^3]+o(1).
\end{equation}
On the other hand, since $\varphi(s_{\delta_n}^3)=0$, $\varphi(s_{\delta_n})=0$ and $\mathbb{E}\tilde{N}_n=0$, we have 
\begin{align*}
    \tilde{\varphi}[(\delta_n\tilde{N}_n\cdot1_\mathcal{A}+s_{\delta_n})^3]&=\tilde{\varphi}[\delta_n^3\tilde{N}_n^3\cdot1_\mathcal{A}+s_{\delta_n}^3+3\delta_n^2\tilde{N}_n^2s_{\delta_n}+3\delta_n\tilde{N}_ns_{\delta_n}^2]\\
&=\delta_n^3\mathbb{E}\tilde{N}_n^3+\varphi(s_{\delta_n}^3)+3\delta_n^2\mathbb{E}\tilde{N}_n^2\varphi(s_{\delta_n})+3\delta_n\mathbb{E}\tilde{N}_n\varphi(s_{\delta_n}^2)\\
    &=\frac{\delta_n^3}{\gamma_n^3}[\mathbb{E}N_n^3-\eta_n^3-3\eta_n\mathbb{E}N_n^2+3\eta_n^2\mathbb{E}N_n]\\
    &=\frac{\delta_n^3}{\gamma_n^3}[\mathbb{E}N_n^3+2\eta_n^3-3\eta_n\rho_n^2],
\end{align*}
which on comparing with (\ref{thm31pf1}), shows that $\tilde{\varphi}(\tilde{\mathcal{S}}_{N_n}^3)=\tilde{\varphi}[(\delta_n\tilde{N}_n\cdot1_\mathcal{A}+s_{\delta_n})^3]+o(1)$. This verifies the theorem for $m=1, 2, 3$. 
    


\begin{proof}[\textbf{Proof of Theorem \ref{thm31}(i)}]
Let $\mathcal{S}_i'=\sum_{j=1}^{i}(a_j-\alpha)$, $i\in\mathbb{N}$. For $m\ge1$, 
\begin{align*}
    \mu_m=\tilde{\varphi}(\tilde{\mathcal{S}}_{N_n}^m(\delta_n\tilde{N}_n)^l)&=\sum_{i=0}^{\infty}\mathbb{P}\{N_n=i\}(\delta_n\frac{i-\eta_n}{\gamma_n})^l\varphi\Big[(\frac{\mathcal{S}_i-\eta_n\alpha}{\sigma_n})^m\Big]\\
    &=\sum_{i=0}^{\infty}\mathbb{P}\{N_n=i\}(\delta_n\frac{i-\eta_n}{\gamma_n})^l\varphi\Big[(\frac{\mathcal{S}_i-i\alpha+(i-\eta_n)\alpha}{\sigma_n})^m\Big]\\
    &=\sum_{i=0}^{\infty}\mathbb{P}\{N_n=i\}(\delta_n\frac{i-\eta_n}{\gamma_n})^l\sum_{r=0}^{m}\binom{m}{r}(\frac{\alpha}{\sigma_n})^r(i-\eta_n)^r\varphi\Big[(\frac{\mathcal{S}_i-i\alpha}{\sigma_n})^{m-r}\Big]\\
    &=\sum_{r=0}^{m}\binom{m}{r}\sum_{i=0}^{\infty}\mathbb{P}\{N_n=i\}(\delta_n\frac{i-\eta_n}{\gamma_n})^l(\frac{\alpha}{\sigma_n})^r(i-\eta_n)^r\varphi\Big[(\frac{\mathcal{S}_i'}{\sigma_n})^{m-r}\Big].
\end{align*}
Let us define
\begin{equation}\label{pf11}
    \mu_m'=\sum_{r=0}^{m}\binom{m}{r}\sum_{i=0}^{\infty}\mathbb{P}\{N_n=i\}(\delta_n\frac{i-\eta_n}{\gamma_n})^l(\frac{\alpha}{\sigma_n})^r(i-\eta_n)^r\varphi\Big[(\frac{\mathcal{S}_{ \lfloor \eta_n\rfloor}'}{\sigma_n})^{m-r}\Big],
\end{equation}
where $\lfloor\cdot\rfloor$ denotes the greatest integer function. Then,
\begin{align}
    \mu_m-\mu_m'&=\sum_{r=0}^{m}\binom{m}{r}\sum_{i=0}^{\infty}\mathbb{P}\{N_n=i\}(\delta_n\frac{i-\eta_n}{\gamma_n})^l(\frac{\alpha}{\sigma_n})^r(i-\eta_n)^r\Big[\varphi\Big[(\frac{\mathcal{S}_i'}{\sigma_n})^{m-r}\Big]-\varphi\Big[(\frac{\mathcal{S}_{\lfloor\eta_n\rfloor}'}{\sigma_n})^{m-r}\Big]\Big]\nonumber\\
    &=\sum_{i=0}^{\infty}\mathbb{P}\{N_n=i\}(\delta_n\frac{i-\eta_n}{\gamma_n})^l\Big[\varphi\Big[(\frac{\mathcal{S}_i-\eta_n\alpha}{\sigma_n})^m\Big]-\varphi\Big[(\frac{\mathcal{S}_{\lfloor\eta_n\rfloor}+i\alpha-(\lfloor\eta_n\rfloor+\eta_n)\alpha}{\sigma_n})^m\Big]\Big].\label{pf12}
\end{align}
Here, using (\ref{momcumrelation}), we have
\begin{align}
    \varphi\Big[&(\frac{\mathcal{S}_i-\eta_n\alpha}{\sigma_n})^m\Big]-\varphi\Big[(\frac{\mathcal{S}_{\lfloor\eta_n\rfloor}+i\alpha-(\lfloor\eta_n\rfloor+\eta_n)\alpha}{\sigma_n})^m\Big]\nonumber\\
    &=\frac{1}{\sigma_n^m}\sum_{\pi\in NC(m)}[\kappa_\pi(\mathcal{S}_i-\eta_n\alpha)-\kappa_\pi(\mathcal{S}_{\lfloor\eta_n\rfloor}+i\alpha-(\lfloor\eta_n\rfloor+\eta_n)\alpha)]\nonumber\\
    &=\frac{1}{\sigma_n^m}\sum_{\substack{\pi\in NC(m):\\
    |V|=1\ \text{for some}\ V\in\pi}}[\kappa_\pi(\mathcal{S}_i-\eta_n\alpha)-\kappa_\pi(\mathcal{S}_{\lfloor\eta_n\rfloor}+i\alpha-(\lfloor\eta_n\rfloor+\eta_n)\alpha)]\nonumber\\
    &\ \ +\frac{1}{\sigma_n^m}\sum_{\substack{\pi\in NC(m):\\
    |V|\neq1\ \text{for all}\ V\in\pi}}[\kappa_\pi(\mathcal{S}_i-\eta_n\alpha)-\kappa_\pi(\mathcal{S}_{\lfloor\eta_n\rfloor}+i\alpha-(\lfloor\eta_n\rfloor+\eta_n)\alpha)]\nonumber\\
    &=\frac{1}{\sigma_n^m}\sum_{\substack{\pi\in NC(m):\\
    |V|=1\ \text{for some}\ V\in\pi}}\Big[\prod_{V\in\pi:|V|=1}\kappa_1(\mathcal{S}_i-\eta_n\alpha)\prod_{V\in\pi:|V|\neq1}\kappa_{|V|}(\mathcal{S}_i)\nonumber\\
    &\ \ -\prod_{V\in\pi:|V|=1}\kappa_1(\mathcal{S}_{\lfloor\eta_n\rfloor}+i\alpha-(\lfloor\eta_n\rfloor+\eta_n)\alpha)\prod_{V\in\pi:|V|\neq1}\kappa_{|V|}(\mathcal{S}_{\lfloor\eta_n\rfloor})\Big]\nonumber\\
    &\ \ +\frac{1}{\sigma_n^m}\sum_{\substack{\pi\in NC(m):\\
    |V|\neq1\ \text{for all}\ V\in\pi}}[\kappa_\pi(\mathcal{S}_i)-\kappa_\pi(\mathcal{S}_{\lfloor\eta_n\rfloor})]\nonumber\\
    &=\frac{1}{\sigma_n^m}\sum_{\substack{\pi\in NC(m):\\
    |V|=1\ \text{for some}\ V\in\pi,\\
    \pi\neq0_m}}\Big[\prod_{V\in\pi:|V|=1}(i-\eta_n)\alpha\Big]\Big[\prod_{V\in\pi:|V|\neq1}i\kappa_{|V|}(a_1) -\prod_{V\in\pi:|V|\neq1}\lfloor\eta_n\rfloor\kappa_{|V|}(a_1)\Big]\nonumber\\
    &\ \ +\frac{1}{\sigma_n^m}\sum_{\substack{\pi\in NC(m):\\
    |V|\neq1\ \text{for all}\ V\in\pi}}(i^{|\pi|}-\lfloor\eta_n\rfloor^{|\pi|})\kappa_{\pi}(a_1)\label{pf13}\\
    &=\frac{1}{\sigma_n^m}\sum_{\substack{\pi\in NC(m):\\
    |V|=1\ \text{for some}\ V\in\pi,\\
    \pi\neq 0_m}}\alpha^{\#1^\pi}(i-\eta_n)^{\#1^\pi}(i^{|\pi|-\#1^\pi}-\lfloor\eta_n\rfloor^{|\pi|-\#1^\pi})\prod_{V\in\pi:|V|\neq1}\kappa_{|V|}(a_1)\nonumber\\
    &\ \ +\frac{1}{\sigma_n^m}\sum_{\substack{\pi\in NC(m):\\
    |V|\neq1\ \text{for all}\ V\in\pi}}(i^{|\pi|}-\lfloor\eta_n\rfloor^{|\pi|})\kappa_{\pi}(a_1),\label{pf14}
\end{align}
where $\#1^\pi=|\{V\in\pi:\,|V|=1\}|$ is the number of blocks in $\pi$ with only one element. To get (\ref{pf13}), we have used, (i) the freeness of $a_i$'s, (ii) the fact that free cumulants of order greater than one are translation invariant, and (iii)  for $\pi=\textbf{0}_m$, $\kappa_\pi(\mathcal{S}_i-\eta_n\alpha)-\kappa_\pi(\mathcal{S}_{\lfloor\eta_n\rfloor}+i\alpha-(\lfloor\eta_n\rfloor+\eta_n)\alpha)=0$. Now, on substituting (\ref{pf14}) in (\ref{pf12}), we get
\begin{align}
    &\mu_m-\mu_m'\nonumber\\
    &=\frac{1}{\sigma_n^m}\sum_{\substack{\pi\in NC(m):\\
    |V|=1\ \text{for some}\ V\in\pi,\\
    \pi\neq 0_m}}\alpha^{\#1^\pi}\sum_{i=0}^{\infty}\mathbb{P}\{N_n=i\}(\delta_n\frac{i-\eta_n}{\gamma_n})^l(i-\eta_n)^{\#1^\pi}(i^{|\pi|-\#1^\pi}-\lfloor\eta_n\rfloor^{|\pi|-\#1^\pi})\nonumber\\
    &\ \ \cdot\prod_{V\in\pi:|V|\neq1}\kappa_{|V|}(a_1)+\frac{1}{\sigma_n^m}\sum_{\substack{\pi\in NC(m):\\
    |V|\neq1\ \text{for all}\ V\in\pi}}\sum_{i=0}^{\infty}\mathbb{P}\{N_n=i\}(\delta_n\frac{i-\eta_n}{\gamma_n})^l(i^{|\pi|}-\lfloor\eta_n\rfloor^{|\pi|})\kappa_{\pi}(a_1)\nonumber\\
    &=\frac{\alpha^{l+\#1^\pi}}{\sigma_n^{m+l}}\sum_{\substack{\pi\in NC(m):\\
    |V|=1\ \text{for some}\ V\in\pi,\\
    \pi\neq 0_m}}\Big[\prod_{V\in\pi:|V|\neq1}\kappa_{|V|}(a_1)\Big]\Big[\mathbb{E}[(N_n-\eta_n)^{\#1^\pi+l}(N_n^{|\pi|-\#1^\pi}-\lfloor\eta_n\rfloor^{|\pi|-\#1^\pi})\Big]\nonumber\\
    &\ \ +\frac{\alpha^l}{\sigma_n^{m+l}}\sum_{\substack{\pi\in NC(m):\\
    |V|\neq1\ \text{for all}\ V\in\pi}}\kappa_{\pi}(a_1)\mathbb{E}[(N_n^{|\pi|}-\lfloor\eta_n\rfloor^{|\pi|})(N_n-\eta_n)^l],\label{pf15}
\end{align}
where we have used $\delta_n/\gamma_n=\alpha/\sigma_n$ to get the last step. Now, consider 
\begin{align}
    \frac{1}{\sigma_n^{m+l}}\mathbb{E}[(N_n&-\eta_n)^{\#1^\pi+l}(N_n^{|\pi|-\#1^\pi}-\lfloor\eta_n\rfloor^{|\pi|-\#1^\pi})]\nonumber\\
    &= \frac{1}{\sigma_n^{m+l}}\mathbb{E}[(N_n-\eta_n)^{\#1^\pi+l}(N_n-\eta_n+\eta_n)^{|\pi|-\#1^\pi}]-\mathbb{E}[(N_n-\eta_n)^{\#1^\pi+l}]\lfloor\eta_n\rfloor^{|\pi|-\#1^\pi}\nonumber\\
    &= \frac{1}{\sigma_n^{m+l}}\sum_{r=1}^{|\pi|-\#1^\pi}\binom{|\pi|-\#1^\pi}{r}\mathbb{E}[(N_n-\eta_n)^{\#1^\pi+l+r}]\eta_n^{|\pi|-\#1^\pi-r}\nonumber\\
    &\ \ + \frac{1}{\sigma_n^{m+l}}\mathbb{E}[(N_n-\eta_n)^{\#1^\pi+l}](\eta_n^{|\pi|-\#1^\pi}-\lfloor\eta_n\rfloor^{|\pi|-\#1^\pi}).\label{pf1c1}
\end{align}
Note that for any $\pi\neq 0_m$, $0\leq \#1^\pi< m$ and $|\pi|-\#1^\pi\leq (m-\#1^\pi)/2$, which implies $\#1^\pi+2(|\pi|-\#1^\pi)\leq m$. So, using the third condition of (\ref{thm31cond1}) and $\eta_n\leq \sigma_n^2/\theta^2$, the first term in (\ref{pf1c1}) reduces to
\begin{align*}
    \frac{1}{\sigma_n^{m+l}}\sum_{r=1}^{|\pi|-\#1^\pi}\binom{|\pi|-\#1^\pi}{r}&\mathbb{E}[(N_n-\eta_n)^{\#1^\pi+l+r}]\eta_n^{|\pi|-\#1^\pi-r}\\
    &=\frac{1}{\sigma_n^{m+l}}\sum_{r=1}^{|\pi|-\#1^\pi}\binom{|\pi|-\#1^\pi}{r}O(\sigma_n^{2|\pi|+l-\#1^\pi-r})=o(1).
\end{align*}
Moreover, for $|\pi|-\#1^\pi=1$, $0\leq \eta_n-\lfloor\eta_n\rfloor<1$ and $\eta_n^{|\pi|-\#1^\pi}-\lfloor\eta_n\rfloor^{|\pi|-\#1^\pi}=O(\lfloor\eta_n\rfloor^{|\pi|-\#1^\pi-1})$ whenever $|\pi|-\#1^\pi>1$. In view of this, again using the third condition of (\ref{thm31cond1}) and $\eta_n\leq \sigma_n^2/\theta^2$, the second term in (\ref{pf1c1}) is
\begin{equation*}
    \frac{1}{\sigma_n^{m+l}}\mathbb{E}[(N_n-\eta_n)^{\#1^\pi+l}](\eta_n^{|\pi|-\#1^\pi}-\lfloor\eta_n\rfloor^{|\pi|-\#1^\pi})=o(1).
\end{equation*}
Thus, (\ref{pf1c1}) reduces to
\begin{equation}\label{pf1c2}
    \frac{1}{\sigma_n^{m+l}}\mathbb{E}[(N_n-\eta_n)^{\#1^\pi+l}(N_n^{|\pi|-\#1^\pi}-\lfloor\eta_n\rfloor^{|\pi|-\#1^\pi})]=o(1).
\end{equation}
 
Furthermore, for $\pi\in NC(m)$ such that $|V|\neq 1$ for each $V\in\pi$, we have $|\pi|\leq m/2$. Thus, using a similar argument as for (\ref{pf1c2}), we get
\begin{equation}\label{pf1c3}
    \frac{1}{\sigma_n^{m+l}}\mathbb{E}[(N_n^{|\pi|}-\lfloor\eta_n\rfloor^{|\pi|})(N_n-\eta_n)^l]=o(1).
\end{equation}
Substituting (\ref{pf1c2}) and (\ref{pf1c3}) in (\ref{pf15}) yields
\begin{equation}\label{pf17}
    \mu_m=\mu_m'+o(1).
\end{equation}

We now derive the approximation of $\mu_m'$, as $n\rightarrow\infty$. From (\ref{pf11}),
\begin{align}
    \mu_m'&=\sum_{r=0}^{m}\binom{m}{r}\sum_{i=0}^{\infty}\mathbb{P}\{N_n=i\}(\frac{\alpha}{\sigma_n})^r(\delta_n\frac{i-\eta_n}{\gamma_n})^l(i-\eta_n)^r\varphi\Big[(\frac{\mathcal{S}_{\lfloor\eta_n\rfloor}'}{\sigma_n})^{m-r}\Big]\nonumber\\
    &=\sum_{r=0}^{m}\binom{m}{r}\sum_{i=0}^{\infty}\mathbb{P}\{N_n=i\}(\frac{\alpha}{\sigma_n})^r(i-\eta_n)^r(\delta_n\frac{i-\eta_n}{\gamma_n})^l(\frac{\sqrt{\eta_n}\theta}{\sigma_n})^{m-r}\varphi\Big[(\frac{\mathcal{S}_{\lfloor\eta_n\rfloor}'}{\sqrt{\eta_n}\theta})^{m-r}\Big].\label{pf18}
\end{align}
From (\ref{deltadef}),
\begin{equation}\label{pf19}
    \frac{\sqrt{\eta_n}\theta}{\sigma_n}=(\frac{\eta_n\theta^2}{\sigma_n^2})^{\frac{1}{2}}=(1-\delta_n^2)^{\frac{1}{2}}.
\end{equation}
As $\eta_n\rightarrow\infty$, in view of the free clt (see Theorem 8.10 of \cite{Nica2007}), 
\begin{equation}\label{pf110}
    \varphi\Big[(\frac{\mathcal{S}_{\lfloor\eta_n\rfloor}'}{\sqrt{\eta_n}\theta})^{m-r}\Big]=\varphi[s^{m-r}]+o(1),
\end{equation}
where $s$ is a standard semi-circle variable. Also, using $\eta_n\theta^2\leq \sigma_n^2$, we have 
\begin{align}
    \Big|\sum_{r=0}^{m}\binom{m}{r}\sum_{i=0}^{\infty}\mathbb{P}\{N_n=i\}&(\frac{\alpha}{\sigma_n})^r(i-\eta_n)^r(\delta_n\frac{i-\eta_n}{\gamma_n})^l(\frac{\sqrt{\eta_n}\theta}{\sigma_n})^{m-r}\Big|\nonumber\\
    &\leq \sum_{r=0}^{m}\binom{m}{r}\mathbb{E}[|N_n-\eta_n|^r](\delta_n\frac{|N_n-\eta_n|}{\gamma_n})^l(\frac{\alpha}{\sigma_n})^r\nonumber\\
    &= \sum_{r=0}^{m}\binom{m}{r}\frac{O(\sigma_n^{r+l})}{\sigma_n^{r+l}}\alpha^{r+l}\leq C<\infty,\label{pf111}
\end{align}
where the last step follows from (\ref{thm31cond1}).
Therefore, using (\ref{pf19})-(\ref{pf111})
 in (\ref{pf18}), we get
\begin{align*}
    \mu_m'&=\sum_{r=0}^{m}\binom{m}{r}\sum_{i=0}^{\infty}\mathbb{P}\{N_n=i\}(\frac{\alpha}{\sigma_n})^r(i-\eta_n)^r(\delta_n\frac{i-\eta_n}{\gamma_n})^l(\frac{\sqrt{\eta_n}\theta}{\sigma_n})^{m-r}\varphi[s^{m-r}]+o(1)\\
    &=\sum_{r=0}^{m}\binom{m}{r}\sum_{i=0}^{\infty}\mathbb{P}\{N_n=i\}(\frac{\delta_n}{\gamma_n})^r(i-\eta_n)^r(\delta_n\frac{i-\eta_n}{\gamma_n})^l(1-\delta_n^2)^{(m-r)/2}\varphi[s^{m-r}]+o(1)\\
    &=\sum_{r=0}^{m}\binom{m}{r}\delta_n^r\mathbb{E}(\frac{N_n-\eta_n}{\gamma_n})^r(\delta_n\frac{N_n-\eta_n}{\gamma_n})^l(1-\delta_n^2)^{(m-r)/2}\varphi[s^{m-r}]+o(1)\\
    &=\mathbb{E}(\delta_n\frac{N_n-\eta_n}{\gamma_n})^l\varphi\Big[(\delta_n\frac{N_n-\eta_n}{\gamma_n}\cdot1_\mathcal{A}+\sqrt{1-\delta_n^2}s)^m\Big]+o(1),
\end{align*}
which on using in (\ref{pf17}) yields (\ref{thm31re}). This completes the proof of Theorem \ref{thm31} (i).  
\end{proof}

\begin{proof} [\textbf{Proof of Theorem \ref{thm31}(ii)}]
Let us, without loss, assume that $\eta_n$ is an integer; otherwise, as in the proof of Part (i), the argument below goes through if $\eta_n$ is replaced by its integer part. For $m,l\in\mathbb{N}\cup\{0\}$, consider
\begin{align}
        \mu_m''&:=\sum_{i=0}^{\infty}\mathbb{P}\{N_n=i\}(\frac{i}{\eta_n})^l\varphi\Big[(\frac{\sqrt{i}\mathcal{S}_{\eta_n}}{\eta_n\theta})^m\Big]\nonumber\\
        &=\sum_{i=0}^{\infty}\mathbb{P}\{N_n=i\}(\frac{i}{\eta_n})^{l+\frac{m}{2}}\varphi\Big[(\frac{\mathcal{S}_{\eta_n}}{\sqrt{\eta_n}\theta})^m\Big]\nonumber\\
        &=\mathbb{E}\Big[(\frac{N_n}{\eta_n})^{l+\frac{m}{2}}\Big]\varphi\Big[(\frac{\mathcal{S}_{\eta_n}}{\sqrt{\eta_n}\theta})^m\Big].\label{pf3.21}
    \end{align}
    From (\ref{pf110}), we have
    \begin{equation}\label{pf3.22}
        \varphi\Big[(\frac{\mathcal{S}_{\eta_n}}{\sqrt{\eta_n}\theta})^m\Big]=\varphi(s^m)+o(1),
    \end{equation}
    where $s$ is the standard semi-circle variable. Thus, using (\ref{thm31cond2}), we get
    \begin{equation*}
        \mu_m''=\mathbb{E}\Big[(\frac{N_n}{\eta_n})^{l+\frac{m}{2}}\Big]\varphi(s^m)+o(1).
    \end{equation*}

    
To complete the proof, it is now sufficient to show that, as $n\rightarrow\infty$,
\begin{equation*}
   \Big|\tilde{\varphi}\Big[\tilde{\mathcal{S}}_{N_n}^m(\frac{N_n}{\eta_n})^l\Big]-\mu_m''\Big|=o(1).
\end{equation*}
 Recall that $\alpha=0$. Hence we have
    \begin{equation}\label{pf3.25}
        \tilde{\varphi}\Big[\tilde{\mathcal{S}}_{N_n}^m(\frac{N_n}{\eta_n})^l\Big]-\mu_m''=\sum_{i=0}^{\infty}\mathbb{P}\{N_n=i\}(\frac{i}{\eta_n})^l\Big[\varphi\Big[(\frac{\mathcal{S}_i}{\sqrt{\eta_n}\theta})^m\Big]-\varphi\Big[(\frac{\sqrt{i}\mathcal{S}_{\eta_n}}{\eta_n\theta})^m\Big]\Big].
    \end{equation}
    Let $NC_2^c(m)$ denote the complement of $NC_2(m)$. From (\ref{momcumrelation}), we have
    \begin{align}
        \varphi\Big[(\frac{\mathcal{S}_i}{\sqrt{\eta_n}\theta})^m\Big]-\varphi\Big[(\frac{\sqrt{i}\mathcal{S}_{\eta_n}}{\eta_n\theta})^m\Big]&=\sum_{\pi\in NC(m)}\Big[(\frac{1}{\sqrt{\eta_n}\theta})^m\kappa_\pi(\mathcal{S}_i)-(\frac{\sqrt{i}}{\eta_n\theta})^m\kappa_\pi(\mathcal{S}_{\eta_n})\Big]\nonumber\\
        &=\sum_{\substack{\pi\in NC(m):\\
        |V|\neq1\ \text{for each}\ V\in\pi}}\Big[(\frac{1}{\sqrt{\eta_n}\theta})^m i^{|\pi|}-(\frac{\sqrt{i}}{\eta_n\theta})^m\eta_n^{|\pi|}\Big]\kappa_\pi(a_1)\nonumber\\
        &=\sum_{\substack{\pi\in NC(m)\cap NC_2^c(m):\\
        |V|\neq1\ \text{for each}\ V\in\pi}}\Big[(\frac{1}{\sqrt{\eta_n}\theta})^m i^{|\pi|}-(\frac{\sqrt{i}}{\eta_n\theta})^m\eta_n^{|\pi|}\Big]\kappa_\pi(a_1),\label{pf3.26}
    \end{align}
    where we have used the freeness of $a_j$'s and $\varphi(a_j)=0$ to get the penultimate step, and the last step follows by using the fact that $|\pi|=m/2$ for all $\pi\in NC_2(m)$, whenever $m$ is even. Then, on substituting (\ref{pf3.26}) in (\ref{pf3.25}), we get
    \begin{align*}
        \Big|\tilde{\varphi}\Big[\tilde{\mathcal{S}}_{N_n}^m(\frac{N_n}{\eta_n})^l\Big]-\mu_m''\Big|&=\Big|\sum_{i=0}^{\infty}\mathbb{P}\{N_n=i\}(\frac{i}{\eta_n})^l\sum_{\substack{\pi\in NC(m)\cap NC_2^c(m):\\
        |V|\neq1\ \text{for each}\ V\in\pi}}\Big[(\frac{1}{\sqrt{\eta_n}\theta})^m i^{|\pi|}-(\frac{\sqrt{i}}{\eta_n\theta})^m\eta_n^{|\pi|}\Big]\kappa_\pi(a_1)\Big|\\
        &\leq \frac{1}{\theta^m}\sum_{\substack{\pi\in NC(m)\cap NC_2^c(m):\\
        |V|\neq1\ \text{for each}\ V\in\pi}}\Big[\mathbb{E}\Big[|\frac{N_n}{\eta_n}|^{|\pi|+l}\Big]+\mathbb{E}\Big[|\frac{N_n}{\eta_n}|^{l+\frac{m}{2}}\Big]\Big]\eta_n^{|\pi|-\frac{m}{2}}\kappa_\pi(a_1)\\
        &=o(1),
    \end{align*}
    where we have used (\ref{thm31cond2}), and the fact that $|\pi|<m/2$ whenever $\pi\in  NC(m)\cap NC_2^c(m)$, to obtain the last step. This completes the proof of Theorem \ref{thm31} (ii).
\end{proof}
\section{Random sums under other forms of independence}\label{sec5}
Now we consider random sums of self-adjoint variables under two different dependence structures, namely, independence and half-independence, and establish results analogous to those obtained in Section \ref{sec4}. 
The corresponding limits here involve the Gaussian and symmetric Rayleigh variables, instead of the semi-circle variable. 
\vskip5pt

\noindent \textbf{Independent variables.} Recall the NCP $(\tilde{\mathcal{A}}, \tilde{\varphi})$. 
As before $\{a_j\}$
are variables from $\mathcal{A}$ and 
$N_n$, $n\ge1$ is a non-negative integer valued rv (independent of $\mathcal{A}$) 
and 
   $\mathcal{S}_{N_n}\coloneqq a_1+\dots+a_{N_n}.$
But now we assume that $\{a_j\}$ are independent (necessarily commutative) variables. 

We have the following moment version of Theorem \ref{thmjtcov} similar to Theorem \ref{thm31}. We omit the proof, which, in view of Theorem \ref{thm:generalclt} (ii), would be along the same lines as the proof of Theorem \ref{thm31}, except we have to use the lattice of all partitions, cumulants and moment-cumulant relations, instead of non-crossing partitions, free cumulants, and moment--free cumulant relations. 

\begin{theorem} \label{thm51}
Let $\mathcal{S}_{N_n}$ be as in Theorem \ref{thm31}, except that its summands $\{a_i\}$ are independent and identically distributed variables on an NCP $(\mathcal{A},\varphi)$. Let $z$ be a standard Gaussian variable in $\tilde{\mathcal{A}}$, which is independent of $\tilde{N}_n$ (as defined above).
    Then, we have the following: \vskip5pt

\noindent(i) Under the assumptions of Theorem \ref{thm31} (i), we have \begin{equation*}
        \tilde{\varphi}(\tilde{\mathcal{S}}_{N_n}^m(\delta_n\tilde{N}_n)^l)=\tilde{\varphi}[(\delta_n\tilde{N}_n\cdot1_\mathcal{A}+\sqrt{1-\delta_n^2}z)^m(\delta_n\tilde{N}_n)^l]+o(1)\ \ \text{for all}\ \ m,l\in\mathbb{N}\cup\{0\}.
    \end{equation*}
    \vskip5pt
      
\noindent (ii) Under the assumptions of Theorem \ref{thm31} (ii), 
    \begin{equation*}
         \tilde{\varphi}(\tilde{\mathcal{S}}_{N_n}^m(\frac{N_n}{\eta_n})^l)=\mathbb{E}[(\frac{N_n}{\eta_n})^{\frac{m}{2}+l}]\varphi(z^m)+o(1),\  \ \text{for all}\ \ m,l\in\mathbb{N}\cup\{0\}.
        \end{equation*}
    \end{theorem}

In view of Theorem \ref{thm51}, results analogous to Corollaries \ref{corr41}-\ref{corr45} can also be formulated and proved for independent variables. We omit the details. In this setting, the limiting variable involves $z$ instead of $s$. 
\begin{remark}
    If $\tilde{N}_n$ is asymptotically Gaussian, then under the assumptions of both Corollaries \ref{corr32} and \ref{corr33}, $\tilde{\mathcal{S}}_{N_n}$ in Theorem \ref{thm51} converges in $*$-distribution to $z$. 
\end{remark}
\vskip5pt

\noindent {\bf Half independent variables.}
The notion of half-independence was defined in \cite{Banica2012}. In \cite{Bose2011}, the relation of half independence with half cumulants was obtained. Let $I$ be an indexing set. Non-commuting variables  $\{a_i\}_{i\in I}\subset\mathcal{A}$ are called \textit{half commuting} if $a_ia_ja_k=a_ka_ja_i$ for any $i,j,k\in I$. This implies $\{a_i^2\}_{i\in I}$ commute. Any monomial $a=a_{i_1}a_{i_2}\dots a_{i_k}$ for $a_{i_r}\in\{a_i\}_{i\in I}$ is called \textit{symmetric} with respect to $\{a_i\}_{i\in I}$ if each variable $a_{i_r}$ appears an equal number of times at odd and even positions in $a$. If a monomial is not symmetric, then it is called \textit{non-symmetric}. Variables 
$\{a_i\}_{i\in I}$ are called \textit{half independent} if $\{a_i^2\}_{i\in I}$
are independent, and for any non-symmetric monomial $a_{i_1}a_{i_2}\dots a_{i_k}$ with respect to $\{a_i\}_{i\in I}$, $\varphi(a_{i_1}a_{i_2}\dots a_{i_k})=0$. The concept of half independence does not extend to sub-algebras.

Any subset $A\subset\mathbb{N}$ is called \textit{symmetric} if it contains an equal number of odd and even integers. Clearly, a symmetric set is of even cardinality. A partition of $A$ is called symmetric if all its blocks are symmetric. Let $E(2m)$ denote the set of symmetric partitions of $\{1,2,\dots,2m\}$, $m\in\mathbb{N}$.

Let $\{a_i\}_{i\in I}$ be variables in an NCP $(\mathcal{A},\varphi)$ such that $\varphi(a_{i_1}\dots a_{i_m})=0$ for all odd $m$ and $i_1,\dots, i_k\in I$. The half cumulants of $\{a_i\}_{i\in I}$, denoted by $\{h_m\}_{m\ge1}$ are defined by the following moment cumulant relation:
\begin{equation}\label{momhalfcum}
    \varphi(a_{i_1}\dots a_{i_m})=\sum_{\pi\in E(m)}h_\pi(a_{i_1},\dots, a_{i_m}),
\end{equation}
 where $\{h_\pi\}$ is the multiplicative extension of $\{h_m\}$. For any self-adjoint variable $a\in\mathcal{A}$, its $m$th cumulant is $h_m=h_m(a,\dots, a)$. Note that the half cumulants of odd order are zero. Let $a=\{a_i\}_{1\leq i\leq n}$ be self-adjoint variables and $k_j$ denote the number of times $a_j$ appears in the monomial $a_{i_1}\dots a_{i_{2m}}$. If these are half independent, then for any symmetric monomial $a$ with respect to $\{a_i\}_{1\leq i\leq n}$, whenever $k_i,k_j\ge2$ for some $1\leq i,j\leq n$,
$ h_{2m}(a_{i_1}\dots a_{i_{2m}})=0$.

In view of Theorem \ref{thm:generalclt} (iii), the following result holds. Using 
(\ref{momhalfcum}), its proof is along lines similar to that of  Theorem \ref{thm31}, and is omitted. 

\begin{theorem} 
Suppose $\{a_i\}$ of $\mathcal{S}_{N_n}$ in Theorem \ref{thm31} are half independent and identically distributed. Let $r_S$ be a standard symmetric Rayleigh variable in $\tilde{\mathcal{A}}$, independent of $\tilde{N}_n$. Then the following hold. \vskip5pt
\noindent(i) Under the conditions of Theorem \ref{thm31} (i), \begin{equation}\label{thm31classical}
        \tilde{\varphi}(\tilde{\mathcal{S}}_{N_n}^m(\delta_n\tilde{N}_n)^l)=\tilde{\varphi}[(\delta_n\tilde{N}_n\cdot1_\mathcal{A}+\sqrt{1-\delta_n^2}r_S)^m(\delta_n\tilde{N}_n)^l]+o(1)\ \ \text{for all}\ \ m,l\in\mathbb{N}\cup\{0\}.
    \end{equation}
      
\noindent (ii) Under the assumptions of Theorem \ref{thm31} (ii), 
    \begin{equation*}
        \tilde{\varphi}(\tilde{\mathcal{S}}_{N_n}^m(\frac{N_n}{\eta_n})^l)=\mathbb{E}[(\frac{N_n}{\eta_n})^{\frac{m}{2}+l}]\varphi(r_S^m)+o(1),\ \  \text{for all}\ \ m,l\in\mathbb{N}\cup\{0\}.
        \end{equation*}
    \end{theorem}



Convergence results analogous to  Corollaries \ref{crr31}-\ref{corr33} hold for random sums of half independent variables under similar conditions.
The limit variable is now replaced by $r_S$.

\section{$*$-convergence of randomly indexed variable}\label{sec6}
In \cite{Anscombe1952} Anscombe established an asymptotic distribution result for a randomly observed rv in the context of sequential estimation.
A consequence of this result is a clt for a random sum of rvs. Here, we present some scaling limit results for randomly picked self-adjoint variables.
Let $\{a_n\}_{n\ge1}$ be a sequence of self-adjoint variables from an NCP $(\mathcal{A},\varphi)$, where $\varphi$ is \textit{tracial}. This 
traciality is required because for the next result, we will need the following inequality, which holds under traciality.

For any $a\in \mathcal{A}$, let $|a|^2:=a^*a$. Let $q_1,\dots,q_m\in\mathbb{N}$ be such that $\frac{1}{q_1}+\dots+\frac{1}{q_m}=1$. Then, for any $a_1,\dots,a_m\in\mathcal{A}$, the following 
generalized H{\"o}lder's inequality holds:
\begin{equation}\label{gholder}
    |\varphi(a_1\dots a_m)|\leq \varphi(|a_1|^{q_1})^{\frac{1}{q_1}}\dots\varphi(|a_m|^{q_m})^{\frac{1}{q_m}}.
\end{equation}

Suppose there exists an $\alpha\in\mathbb{R}$, and a sequence $\{K_n\}_{n\ge1}$ of positive numbers such that,
\vskip5pt
\noindent \textbf{Condition C1.} As $n\rightarrow\infty$, 
$(a_n-\alpha)/K_n\stackrel{*}{\to} l$ 
which is self-adjoint in some NCP $(\mathcal{B}, \Bar{\varphi})$. 
\vskip5pt
\noindent \textbf{Condition C2.} For any small $\epsilon, \epsilon'>0$, there exist a large $M\in\mathbb{N}$ such that for any $n>M$ and $m\in\mathbb{N}$,
\begin{equation*}
    \Big|\varphi[(\frac{a_{n'}-a_{n}}{K_n})^m]\Big|<\epsilon, \ \ \text{whenever}\ \ |n'-n|<n\epsilon'.
\end{equation*}
\begin{theorem}\label{thm61}
    Let $\{w_n\}_{n\ge1}$ be an increasing sequence of positive integers with $w_n\rightarrow\infty$. Let $\{N_n\}$ be non-negative integer-valued rvs 
    such that $N_n/w_n\stackrel{\mathbb{P}}{\to} 1$.
    Let $\{a_n\}$ be self-adjoint variables that satisfy Conditions \textbf{C1} and 
    \textbf{C2}. Let $\{(K_{N_n}/K_{w_n})^m\}$ be uniformly integrable for all $m\in\mathbb{N}$. Then,
    $\frac{a_{N_n}-\alpha}{K_{w_n}}\stackrel{*}{\to} l$.
\end{theorem}
\begin{proof}
    Suppose $\epsilon$ and $\epsilon'$ are as in Condition \textbf{C2}. For $m\in\mathbb{N}$, 
    \begin{align}
        \tilde{\varphi}[(\frac{a_{N_n}-\alpha}{K_{w_n}})^m]&=\sum_{i=0}^{\infty}\mathbb{P}\{N_n=i\}\varphi[(\frac{a_i-\alpha}{K_{w_n}})^m]\nonumber\\
        &=\sum_{|i-w_n|\leq w_n\epsilon'}\mathbb{P}\{N_n=i\}\varphi[(\frac{a_i-\alpha}{K_{w_n}})^m]+\sum_{|i-w_n|> w_n\epsilon'}\mathbb{P}\{N_n=i\}\varphi[(\frac{a_i-\alpha}{K_{w_n}})^m].\label{pf511}
    \end{align}
Using (\ref{gholder}), for any $k_1,l_1,\dots,k_m,l_m\geq0$ such that not all $k_i$'s are zero, from Conditions \textbf{C1} and \textbf{C2}, we get
    \begin{equation*}
        \Big|\varphi\Big[(\frac{a_{n'}-a_{w_n}}{K_{w_n}})^{k_1}(\frac{a_{w_n}-\alpha}{K_{w_n}})^{l_1}\dots (\frac{a_{n'}-a_{w_n}}{K_{w_n}})^{k_m}(\frac{a_{w_n}-\alpha}{K_{w_n}})^{l_m}\Big]\Big|<\epsilon,\ \text{whenever}\ |n'-w_n|<w_n\epsilon',
    \end{equation*}
    for sufficiently large $w_n$, where we have used the fact that $w_n$ is increasing. Thus, as $n\rightarrow\infty$, on the set $|i-w_n|\leq w_n\epsilon'$, we have
    \begin{equation}\label{pf512}
        \varphi[(\frac{a_i-\alpha}{K_{w_n}})^m]=\varphi[(\frac{a_i-a_{w_n}+a_{w_n}-\alpha}{K_{w_n}})^m]=\varphi[(\frac{a_{w_n}-\alpha}{K_{w_n}})^m]+ o(1).
    \end{equation}
    Moreover, under Condition \textbf{C1}, there exists a constant $C_m>0$ such that $|\varphi[(\frac{a_i-\alpha}{K_{i}})^m]|\leq C_m$ for each $m\in\mathbb{N}$. As $\{(\frac{K_{N_n}}{K_{w_n}})^m\}$ is uniformly integrable for each $m\in\mathbb{N}$,  for any $\epsilon>0$, there exist a $\epsilon'>0$ such that for any set $B$ with $\mathbb{P}(B)<\epsilon'$, we have $\sup_{n}\mathbb{E}\{(\frac{K_{N_n}}{K_{w_n}})^m\}\mathbb{I}_{B}<\epsilon$. Note that for large $n$, $\mathbb{P}\{|N_n-w_n|>w_n\epsilon'\}<\epsilon$, as $\frac{N_n}{w_n}\stackrel{\mathbb{P}}{\to}1$.
Hence,
    \begin{align}
        \sum_{|i-w_n|> w_n\epsilon'}\mathbb{P}\{N_n=i\}\varphi[(\frac{a_i-\alpha}{K_{w_n}})^m]&\leq C_m\sum_{|i-w_n|> w_n\epsilon'}\mathbb{P}\{N_n=i\}(\frac{K_i}{K_{w_n}})^m\nonumber\\
        &=C_m\mathbb{E}\Big[(\frac{K_{N_n}}{K_{w_n}})^m\Big]\mathbb{I}\{|N_n-w_n|>w_n\epsilon'\}\nonumber\\
        &=o(1).\label{pf513}
    \end{align}
    On substituting (\ref{pf512}) and (\ref{pf513}) in (\ref{pf511}), we get
    \begin{equation*}
        \tilde{\varphi}[(\frac{a_{N_n}-\alpha}{K_{w_n}})^m]=\varphi[(\frac{a_{w_n}-\alpha}{K_{w_n}})^m]+ o(1).
    \end{equation*}
    As $w_n$ is increasing to $\infty$, the proof follows using  Condition \textbf{C1}.
\end{proof}

The following result is an immediate consequence of Theorem \ref{thm61}.
\begin{corollary}\label{crrr61}
    Let $\{a_n\}$ be self-adjoint variables, satisfying Condition \textbf{C2} with $K_n=1$, such that $a_n\stackrel{*}{\to} l$.
    Let $N_n$ be non-negative rvs such that $N_n/n\stackrel{\mathbb{P}}{\to}1$, as $n\rightarrow\infty$.
    Then, $a_{N_n}\stackrel{*}{\to} l$.
\end{corollary}

The following is a clt-type result for a random sum of free self-adjoint variables.
\begin{corollary}\label{cr61}
    Let $\{b_i\}$ be a sequence of self-adjoint variables with $\varphi(b_i)=0$, $\varphi(b_i^2)=1$, and $\sup_{i}|\varphi(b_i^m)|<\infty$ for each $m\in\mathbb{N}$. Set $a_n=b_1+\dots+b_n$ for each $n\in\mathbb{N}$. Let $\{w_n\}$ be a sequence of positive integers increasing to infinity. Let $\{N_n\}$ be non-negative integer valued rvs such that 
    $N_n/w_n\stackrel{\mathbb{P}}{\to}1$,
    and $\{(\frac{N_n}{w_n})^{m}\}$ be uniformly integrable for each $m\in\mathbb{N}$. 
    Then the following hold. \vskip5pt
    
    \noindent (i) If $\{b_i\}$ are free then 
    $\frac{a_{N_n}}{\sqrt{w_n}}\stackrel{*}{\to} s$.
    \vskip5pt
    \noindent (ii) If $\{b_i\}$ are independent then 
    $\frac{a_{N_n}}{\sqrt{w_n}}\stackrel{*}{\to} z$.
    \vskip5pt
    \noindent (iii) If $\{b_i\}$ are half-independent then 
    $\frac{a_{N_n}}{\sqrt{w_n}}\stackrel{*}{\to} r_S$.
\end{corollary}
\begin{proof}
    Theorem \ref{thm:generalclt} implies that $a_n/\sqrt{n}$ in (i), (ii), and (iii) converges in $*$-distribution to $s$, $z$ and $r_S$, respectively. Moreover, using the respective moment cumulant formula along with the joint cumulants properties, it can be easily shown that $\{a_n\}$ with $\{K_n=\sqrt{n}\}$ satisfies Condition \textbf{C2} for all three cases.
    Thus, the proof follows from Theorem \ref{thm61}.
\end{proof}

\section{Random sum with random scaling}\label{sec7}
 The free clt with random scaling is easy to derive. 

\begin{theorem}\label{thm71}
    Let $\{a_n\}$ be self-adjoint identically distributed  variables in an NCP $(\mathcal{A}, \varphi)$ with $\varphi(a_1)=0$ and $\varphi(a_1^2)=1$, and $\sup_{n}|\varphi(a_n^k)|<\infty$ for each $k\in\mathbb{N}$. Set $\mathcal{S}_k=a_1+\dots+a_k$, $k\in\mathbb{N}$. Let $\{N_n\}$ be positive integer-valued rvs, independent of $\{a_n\}$, such that $N_n/n\stackrel{\mathbb{P}}{\to}K>0$, a constant, as $n\rightarrow\infty$.\vskip5pt
    
    \noindent (i) If $\{a_i\}$ are free then 
    $\frac{\mathcal{S}_{N_n}}{\sqrt{N_n}}
    \stackrel{*}{\to} s$.
\vskip 3pt
    \noindent (ii) If $\{a_i\}$ are independent then $\frac{\mathcal{S}_{N_n}}{\sqrt{N_n}}\stackrel{*}{\to} z$. 
    \vskip3pt
    \noindent (iii) If $\{a_i\}$ are half independent then $\frac{\mathcal{S}_{N_n}}{\sqrt{N_n}}\stackrel{*}{\to} r_S$. 

   If $a_1$ has a probability measure, then the probability measures of $\frac{\mathcal{S}_{N_n}}{\sqrt{N_n}}$ in (i), (ii) and (iii) exist, and converge weakly to the probability measure of the standard 
    semi-circle, standard Gaussian and standard symmetric Rayleigh measures, respectively. 
\end{theorem}
\begin{proof}
    (i) For $m\in\mathbb{N}$, and for any $\epsilon>0$,
    \begin{align*}
        \tilde{\varphi}[(\frac{\mathcal{S}_{N_n}}{\sqrt{N_n}})^m]&=\sum_{i=1}^{\infty}\mathbb{P}\{N_n=i\}\varphi[(\frac{\mathcal{S}_{i}}{\sqrt{i}})^m]\\
        &=\sum_{i:|i-nK|>n\epsilon}\mathbb{P}\{N_n=i\}\varphi[(\frac{\mathcal{S}_{i}}{\sqrt{i}})^m]+\sum_{i:|i-nK|\leq n\epsilon}\mathbb{P}\{N_n=i\}\varphi[(\frac{\mathcal{S}_{i}}{\sqrt{i}})^m]\\
        &=T_1+T_2 \ \ \text{(say)}.
    \end{align*}
    Note that when $|i-nK|\leq n\epsilon$, $i$ and $n$ are of the same order of magnitude. Thus, as $n\rightarrow\infty$, from the free clt (see (\ref{pf110})), we have $\varphi[(\frac{\mathcal{S}_{i}}{\sqrt{i}})^m]=\varphi(s^m)+o(1)$, where $s$ is a standard semi-circle variable. Moreover, there exists a constant $K'>0$ such that $|\varphi[(\frac{\mathcal{S}_{i}}{\sqrt{i}})^m]|\leq K'$. Hence,
                \[T_1\leq K' \mathbb{P}\{|N_n-nK|>n\epsilon\}\to 0.\]
   On the other hand, using the fact that $N_n/n\stackrel{\mathbb{P}}{\to} K$, 
   \[T_2=\varphi(s^m)[\mathbb{P}\{|N_n-nK|\leq n\epsilon\}+o(1)]\to \varphi(s^m).\]
   This completes the proof of (i). In view of Theorem \ref{thm:generalclt} (ii) and (iii), the proofs of (ii) and (iii) follow similar lines to that of Part (i).
   
   Moreover, observe that the probability measure of $\frac{\mathcal{S}_{N_n}}{\sqrt{N_n}}$ does exist. 
   The remaining part follows immediately from Lemma \ref{lem:momentconv}. 
\end{proof}

\begin{remark}
  Let $\{a_i\}$ be a sequence of self-adjoint variables such that $\varphi(a_i)=0$, $\varphi(a_i^2)=1$ and $\sup_i|\varphi(a_i^k)|<\infty$ for each $k\in\mathbb{N}$. Then, from Theorem \ref{thm:generalclt}, $\mathcal{S}_n=(a_1+\dots+a_n)/\sqrt{n}$ converges in $*$-distribution whenever $\{a_i\}$ are free, independent or half independent. Moreover, it is not difficult to check that $\{\mathcal{S}_n\}$ satisfies Condition \textbf{C2} in all three cases, with $K_n=1$.  Thus, if $N_n/n\stackrel{\mathbb{P}}{\to}1$, then Theorem \ref{thm71} can also be obtained as an immediate consequence of Corollary \ref{crrr61}.
  
\end{remark}

The following result gives a random scaling convergence of a self-adjoint variable randomly observed from a sequence of identically distributed variables. Its proof is straightforward. Hence, we omit it. 
\begin{proposition}
    Let $\{a_j\}$ be a sequence of identical variables having a valid common probability measure $\nu$. Let $\{N_n\}$ be positive integer-valued rvs such that $N_n/n\stackrel{\mathbb{P}}{\to} K$. Then, as $n\rightarrow\infty$, the probability measure of $a_{N_n}/\sqrt{N_n}$ 
    converges weakly to $\delta_{\{0\}}$ (the Dirac measure at $0$). 
\end{proposition}

\section{Some examples}\label{sec8}
We now present some examples for the results of Sections \ref{sec4}, \ref{sec5} and \ref{sec6}. 
\begin{example}\label{exmp1} Let $N_n$ be a Poisson rv with mean $n>0$.  Then, 
    $\eta_n=n$, $\gamma_n^2=n$, $\sigma_n^2=n\theta^2+n\alpha^2$, and $\delta_n^2=\alpha^2/(\theta^2+\alpha^2)$. From clt, the probability measure of $\tilde{N}_n$ converges weakly to the standard Gaussian probability measure, $\mu_Z$ say. Moreover, $\{(N_n/n)^m\}$ is uniformly integrable for each $m\in\mathbb{N}$. \vskip5pt

\noindent(i) Suppose the summand $\{a_i\}$ of $\mathcal{S}_{N_n}$ are free and identically distributed. For $\alpha\neq0$, from Corollary \ref{corr32}, $\tilde{\mathcal{S}}_{N_n}$ converges in $*$-distribution to $\sqrt{\alpha^2/(\theta^2+\alpha^2)}z+\sqrt{\theta^2/(\theta^2+\alpha^2)}s$, where $s$ and $z$  are independent. 
  If $\tilde{\mathcal{S}}_{N_n}$ has a probability measure $\nu_n$, then  by Lemma \ref{lem:momentconv}, $\nu_n$ converges weakly to $\nu$ which is the probability measure of $\sqrt{\alpha^2/(\theta^2+\alpha^2)}z+\sqrt{\theta^2/(\theta^2+\alpha^2)}s$.\vskip5pt 
  
  
\noindent(ii) If $\{a_i\}$ are independent then for $\alpha\neq0$, from Theorem \ref{thm51}(i), $\tilde{\mathcal{S}}_{N_n}\stackrel{*}{\to} z$.\vskip5pt 
\noindent(iii) If we take $K_n=n^t$ for $t>0$ and $w_n=n$ in Theorem \ref{thm61}, then it can be easily checked that $\{(\frac{N_n^t}{n^t})^m\}$ is uniformly integrable for each $m\in\mathbb{N}$. Moreover, $N_n/n\stackrel{\mathbb{P}}{\to}1$.
\end{example}

\begin{example}
    Let $N_n\sim Bin (n, p)$, 
    $p\in(0,1)$. Then, $\eta_n=np$, $\gamma_n^2=np(1-p)$ and $\sigma_n^2=np\theta^2+np(1-p)\alpha^2$, and $\tilde{N}_n$ is asymptotically Gaussian. 
    In this case too, all conclusions of Example \ref{exmp1} hold true, with appropriate parameters of the limiting probability measures.
\end{example}

\begin{example}
    For $n\in\mathbb{N}$, let $N_n$ be a rv that takes values $n$ and $2n$, each with probability $1/2$. Then, $\eta_n=3n/2$, $\gamma_n^2=n^2/4$ and $\sigma_n^2=(6n\theta^2+n^2\alpha^2)/4$. Let $\{a_i\}$ be free and identically distributed.\vskip5pt
    \noindent(i) Suppose $\alpha\neq0$. Then, conditions of Theorem \ref{thm31} are satisfied. It is easy to check that $\tilde{N}_n$ weakly converges to a random variable $N$ with $\mathbb{P}\{N=\pm 1\}=1/2$.

As $\eta_n/\gamma_n^2\rightarrow0$, from Corollary \ref{corr33}, it follows that the probability measure $\nu_n$ of $\tilde{\mathcal{S}}_{N_n}$ weakly converges to $(\delta_{\{-1\}}+\delta_{\{1\}})/2$, where $\delta$ denotes the Dirac measure.\vskip5pt 
\noindent(ii) We have $\lim_{n\rightarrow\infty}\frac{\gamma_n}{\eta_n}=1/3$. So, $\frac{N_n}{\eta_n}$ weakly converges to a discrete rv taking values $2/3$ and $4/3$, each with probability $1/2$. Also, $\{(N_n/\eta_n)^m\}$ is uniformly integrable for each $m\in\mathbb{N}$.  For $\alpha=0$, from Theorem \ref{thm31} (ii), we get
\begin{equation*}
    \lim_{n\rightarrow\infty}\tilde{\varphi}(\tilde{\mathcal{S}}_{N_n}^m)=\varphi(s^m)\frac{2^{m/2}+4^{m/2}}{2\cdot3^{m/2}}.
\end{equation*}
Therefore, $\nu_n$ weakly converges to a mixture of a semi-circle probability measure and $(\delta_{\{\sqrt{2/3}\}}+\delta_{\{\sqrt{4/3}\}})/2$.
\end{example}
\begin{remark}
    For all the examples above, conditions of Theorem \ref{thm31} hold, but (\ref{pf21}) is not satisfied. Let $P_n$ be a Poisson random variable with mean $n>0$, and set $N_n=n^2+P_n$. Then, all conditions of Theorem \ref{thm31} as well as condition (\ref{pf21}) hold true. 
\end{remark}

\end{document}